\documentclass[reqno]{amsart}

\usepackage[utf8]{inputenc}
\usepackage{microtype}
\usepackage{amsmath,amsthm,amssymb}
\usepackage[at]{easylist}
\usepackage{constants}
\usepackage{enumerate}
\usepackage{color}
\usepackage[pagebackref=true,colorlinks,citecolor=blue,linkcolor=blue]{hyperref}
\usepackage{doi}
\allowdisplaybreaks

\renewconstantfamily{normal}{symbol=\gamma, format=\arabic}
\newconstantfamily{eps}{symbol=C, format=\arabic}
\newconstantfamily{i}{symbol=\mathcal{I}, format=\arabic}
\newconstantfamily{j}{symbol=\mathcal{J}, format=\arabic}

\theoremstyle{plain}
\newtheorem{thm}{Theorem}[section]

\newtheorem{lemma}[thm]{Lemma}
\newtheorem{corollary}[thm]{Corollary}
\newtheorem{proposition}[thm]{Proposition}
\newtheorem{remark}[thm]{Remark}

\numberwithin{equation}{section}

\newcommand{\field}[1]{\mathbb{#1}}
\newcommand{\R}{\field{R}}

\newcommand{\set}[1]{{\left\{ #1\right\}}}               	
\newcommand{\pa}[1]{{\left(#1\right)}}                  	
\newcommand{\sq}[1]{{\left[#1\right]}}                  	
\newcommand{\Mi}{\mathcal M}
\newcommand{\abs}[1]{\left| #1 \right|}

\newcommand{\norm}[1]{\left\| #1 \right\|}              	

\newcommand{\datam}{\{m,n,\nu,L\}}

\def\Xint#1{\mathchoice
  {\XXint\displaystyle\textstyle{#1}}%
  {\XXint\textstyle\scriptstyle{#1}}%
  {\XXint\scriptstyle\scriptscriptstyle{#1}}%
  {\XXint\scriptscriptstyle\scriptscriptstyle{#1}}%
  \!\int}
\def\XXint#1#2#3{{\setbox0=\hbox{$#1{#2#3}{\int}$}
  \vcenter{\hbox{$#2#3$}}\kern-.5\wd0}}

\def\bint{\Xint-}
\def\dashint{\Xint{\raise4pt\hbox to7pt{\hrulefill}}}
\def\dashiint{\bint\kern-0.15cm\bint}

\def\YYint#1#2#3{\setbox0=\hbox{$#1{#2#3}{\iint}$}
   \vcenter{\hbox{$#2#3$}}\kern-0.5\wd0}

\newcommand{\st}{\,:\,}                                       	

\newcommand{\eps}{\varepsilon}
\renewcommand{\Lambda}{\varLambda}
\renewcommand{\Delta}{\varDelta}
\newcommand{\gm}{\gamma}
\newcommand{\sig}{\sigma}

\newcommand{\data}{\mathrm{data}}

\newcommand{\loc}{\mathrm{loc}}

\DeclareMathOperator{\diver}{div}
\renewcommand{\div}{\diver}
\newcommand{\pl}{\partial}

\DeclareMathOperator*{\esssup}{ess\,sup}
\newcommand{\essup}{\operatornamewithlimits{ess\,sup}}

\newcommand{\ukp}{(u-k)_+}

\newcommand{\umkmp}{(u^m-k^m)_+}
\newcommand{\umkmnp}{(u^m-k^m_j)_+}
\newcommand{\umkmnup}{(u^m-k^m_{j+1})_+}

\DeclareMathOperator{\supp}{spt}

\DeclareMathOperator{\dist}{dist}

\newcommand{\vfield}[1]{\mathbf{#1}}
\newcommand{\A}{\vfield{A}}

\newcommand{\tder}[1]{\frac{\partial #1}{\partial t}}

\newcommand{\spacedot}{\, \cdot \,}

\newcommand{\deq}{\equiv}

\newcommand{\origin}{o}

\newcommand{\Mean}[1]{{(#1)}}

\let\TeXchi\chi
\newbox\chibox
\setbox0 \hbox{\mathsurround0pt $\TeXchi$}
\setbox\chibox \hbox{\raise\dp0 \box 0 }
\def\chi{\copy\chibox}

\newcommand{\dsty}{\displaystyle}
\newcommand{\txty}{\textstyle}

\title[Self-improving property of singular equations]{Self-improving property of the fast diffusion equation}

\author{Ugo Gianazza}
\address{Dipartimento di Matematica ``F. Casorati''\\
 	Universit\`a di Pavia\\
	via Ferrata 1, 27100 Pavia, Italy}
\email[U.~Gianazza]{gianazza@imati.cnr.it}
\author{Sebastian Schwarzacher}
\address{Katedra matematick\'e anal\'yzy\\ 
Matematicko-fyzik\'aln\'\i\ fakulta Univerzity Karlovy\\ 
Soko\-lovsk\'a 83\\ 
186 75 Praha 8, Czech Republic}	
\email{schwarz@karlin.mff.cuni.cz}
\usepackage[rose,frak,operator]{paper_diening}

\begin{document}

\begin{abstract}
We show that the gradient of the $m$-power of a solution to a singular parabolic equation of porous medium-type (also known as fast diffusion equation), satisfies a reverse H\"older inequality in suitable intrinsic cylinders. Relying on an intrinsic Calder\'on-Zygmund covering argument, we are able to prove the local higher integrability of such a gradient for $m\in\left(\frac{(n-2)_+}{n+2},1\right)$. Our estimates are satisfied for a general class of growth assumptions on the non linearity. In this way, we extend the theory for $m\geq 1$ (see \cite{GiaSch16} in the list of references) to the singular case. In particular, an intrinsic metric that depends on the solution itself is introduced for the singular regime.
\end{abstract}

\maketitle

\section{Introduction and main result}
The aim of this paper is to study regularity properties of the gradient of \emph{non-negative} solutions to nonlinear, parabolic, partial differential equations, whose prototype is the \emph{singular porous medium equation} 
\begin{align}\label{PMD-eq: model}\tag{PME}
u_{t} - \Delta u^{m} = u_{t} - \div \pa{mu^{m-1} Du} = 0\qquad 0 < m \leq 1.
\end{align}
%\begin{align}\label{eq:1}
%u_{t} - \Delta u^{m} = u_{t} - \div \pa{mu^{m-1} Du} = 0\qquad m > 0.
%\end{align}
When $m = 1$, the nonlinear behavior disappears and \eqref{PMD-eq: model} reduces to the  standard heat equation. When $m \neq 1$, the equation is quasi-linear and its \emph{modulus of ellipticity} is $u^{m-1}$.
When $m < 1$, this quantity  blows up as $u \to 0$, the diffusion process dominates over the time evolution (i.e. the diffusion speed is very large) and therefore, the equation is said to be \emph{singular}. The high speed of propagation is the reason for the name of the equation, which is often referred to as \emph{fast diffusion equation}.
Nevertheless, in the following we will prefer and regularly use the term {\em singular porous medium equations}.

Equations of this form arise in applications both from Physics and Mathematics. For example, 
when modelling  the anomalous diffusion of hydrogen plasma across a 
purely poloidal octupole magnetic field, \cite{berryman-77,berryman-78,berryman-holland-80} have shown that
the diffusion equation may be written as the one-dimensional \eqref{PMD-eq: model}
\begin{equation}\label{Eq:anomal}
\frac{\partial}{\partial x}\left[n^{-\frac12}\frac{\partial n}{\partial x}\right]=F(x)\frac{\partial n}{\partial t}\qquad
0\le x\le 1,
\end{equation}
where the geometrical factor $F(x)$ is a strictly positive function determined by the octupole geometry.

A singular porous medium model was proposed by Carleman (\cite{Car}) to study the diffusive limit of kinetic equations.
He considered just two types of particles in a one dimensional setting, moving with speeds $c$ and $-c$. If we denote 
the densities respectively with $u$ and $v$, we can write their simple dynamics as
\begin{equation*}
\left\{
\begin{aligned}
&\partial_t u+c\partial_x u=k(u,v)(v-u)\\
&\partial_t v-c\partial_x v=k(u,v)(u-v)
\end{aligned}
\right.
\end{equation*}
for some interaction kernel $k(u,v)\ge0$. In a typical case, one assumes $\dsty k=(u+v)^\alpha c^2$
with $\alpha\in(0,1]$. If we now write down the equations for $\rho=u+v$ and $j=c(u-v)$ and pass to the limit as 
$c=\frac1\eps\to\infty$, we will obtain to the first order in powers of $\eps$ (see \cite{lions-toscani})
\begin{equation}
\partial_t\rho=\frac12\partial_x\left(\frac1{\rho^\alpha}\partial_x\rho\right),
\end{equation}
which is exactly our case with $m=1-\alpha\in[0,1)$ (what exactly means $m=0$ is beyond the purpose of 
this work). 

As for an example coming not from physical modelling, but from pure Mathematics, the study of the Yamabe Flow in $\mathbb R^n$ ($n>2$) leads to consider
\begin{equation}\label{Eq:yamabe}
\partial_t  v=\Delta v^m,\qquad m=\frac{n-2}{n+2}\in(0,1).
\end{equation}

Obviously, our  list is a very partial one, and there are plenty of other examples: the interested reader can refer, 
for example, to \cite{vazquez2}.

In this paper we  are interested in the order of \emph{integrability} of $|Du^{m}|$; we will deal only with the range 
\begin{equation}\label{Eq:m-range}
m \in \Big(\frac{(n-2)_+}{n+2},1\Big),
\end{equation}
(the case $m>1$ has already been dealt with in \cite{GiaSch16}), and we will study a general class of equations 
which have the same structure as \eqref{PMD-eq: model}. Notice that we are not considering the full interval 
$m\in(0,1)$: the lower bound on $m$ is quite typical, when dealing with regularity issues for the singular porous medium equation (see, for example, the discussion in \cite[Chapter~6, Paragraph~21]{DiBGiaVes11}).

Given a  bounded, open set $E\subset \R^{n}$ with $n\geq2$, and $T >0$, let $E_{T} \deq E \times (0,T)$. 
For $m$ as in \eqref{Eq:m-range}, and $f\in L^\infty_{\loc}(E_T)$, we will consider \emph{nonnegative} solutions to
\begin{equation}\label{PME}
u_{t} -  \div \A(x,t,u,Du^m) = f \qquad \text{weakly in $E_T$.}
\end{equation}
The vector field
 \begin{align*}
\A \colon E_T\times \R \times \R^{n} \to \R^{n}
 \end{align*}
 is only assumed to be measurable, and we suppose there exist constants $0 < \nu \leq L < \infty$ such that
\begin{align}\label{PMS-eq:structure}
\begin{aligned}
&\A(x,t,u,\xi) \cdot \xi \geq \nu\, \abs{\xi}^{2} \\
&\abs{\A(x,t,u,\xi)} \leq L\, \abs{\xi}.
\end{aligned} \qquad\qquad\text{ for a.e. }\ (x,t) \in E_T
\end{align}
By \cite[Chapter~3, \S~5]{DiBGiaVes11}, 
the structure conditions \eqref{PMS-eq:structure} ensure that \eqref{PME} is parabolic. 
Notice that the model problem \eqref{PMD-eq: model} corresponds to the case $\nu = L = 1$ and $f\equiv 0$.
%%%%%%%%%%%%%%%%%%%
\subsection{Weak solutions and sub(super)-solutions}
A function
\begin{equation}\label{PMS-eq: reg weak sol}
u \in C^0_{\loc} \pa{0,T; L^{m+1}_{\loc} (E)}\quad \text{with} \quad u^{m} \in L^{2}_{\loc} \pa{0,T ; W^{1,2}_{\loc}(E)}
\end{equation}
is a local, weak sub(super)-solution to \eqref{PME}-\eqref{PMS-eq:structure} 
if satisfies the integral identity
\begin{equation}\label{PMS-eq: weak super-sol}
\iint_{E_{T}} - u \phi_{t} + \A(x,t,u,Du^m) \cdot D\phi \,dxdt \le(\ge) \iint_{E_{T}} f \phi\,dxdt
\end{equation}
for all possible choices of nonnegative test functions $\phi \in C^{\infty}_{0}(E_{T})$. 
This guarantees that all the integrals in \eqref{PMS-eq: weak super-sol} 
are convergent.

A {local, weak solution} to \eqref{PME}-\eqref{PMS-eq:structure} is both a sub- and a super-solution, i.e., it  satisfies the integral identity
\begin{align}\label{PMS-eq: weak sol}
\iint_{E_{T}} - u \phi_{t} + \A(x,t,u,Du^m) \cdot D\phi \,dxdt = \iint_{E_{T}} f \phi\,dxdt
\end{align}
for all possible choices of test functions $\phi \in C^{\infty}_o(E_{T})$.

By a standard \emph{mollification} argument, it is possible to use $u^m$ as test function (otherwise, handling the time derivative of $u$ becomes a delicate issue). Let $\zeta \colon \R \to \R$,
 \begin{align*}
\zeta(s) \deq
\left\{
\begin{aligned}
&\ C \exp \pa{\frac{1}{\abs{s}^{2} - 1}} && \abs{s} < 1\\
&\ 0 && \abs{s} \geq 1
\end{aligned}
\right.
 \end{align*}
be the standard mollifier ($C$ is chosen in order to have $\norm{\zeta}_{L^{1}(\R)} = 1$) and define the \emph{family}
 \begin{align*}
\zeta^{\eps}(s) = \frac{1}{\eps} \zeta \pa{\frac{s}{\eps}}, \qquad \eps > 0.
 \end{align*}
Since we need a \emph{time regularization}, given  $\phi \in C^{\infty}_o( E _{T})$, we consider the family of mollifiers $\set{\zeta^{\eps}}$, with
 \begin{align*}
\eps < \dist \pa{\supp \phi,  E_T},
 \end{align*}
and we set
 \begin{align*}
\phi_{\eps}(x,t) = \pa{\varphi \star \zeta^{\eps}}(x,t) =\int_{\R} \phi(x, t-s) \zeta^{\eps}(s)\,ds.
 \end{align*} 

We insert $\phi_{\eps}$ as test function in \eqref{PMS-eq: weak sol}, change variables and apply Fubini's theorem to obtain
\begin{align}\label{PMS-eq: weak sol mollified}
\iint_{ E _{T}} - u_{\eps} \phi_{t} + \A_{\eps}(x,t,u,Du^m) \cdot D\phi\,dxdt = \iint_{E_{T}} f_\eps \phi\,dxdt,
\end{align}
where the subscript in $u_{\eps}$, $f_\eps$, and $\A_{\eps}$ denotes the mollification with respect to time.

We conclude this introductory section with our main result. 
\begin{thm}[Local higher integrability]\label{thm:main}
Let $u \geq 0$ be a local, weak solution to \eqref{PME}-\eqref{PMS-eq:structure} in $E_T$ for $\frac{(n-2)_+}{n+2}<m<1$, and $f\in L^\infty_{\loc}(E_T)$.
Then, there exists $\eps_{\origin} > 0$, depending only on $n$, $m$, $\nu$, and $L$ of 
\eqref{PMS-eq:structure}, such that
 \begin{align*}
 u^{{m}} \in L^{2 + \eps}_{\loc} \pa{0,T ; W^{1,2 + \eps}_{\loc}(E)} \qquad \forall \eps \in (0, \eps_{\origin}].
 \end{align*}
\end{thm}
Theorem~\ref{thm:main} is a straightforward consequence of local quantitative estimates. We provide two different versions, a first one for standard parabolic cylinders $B_r(x_o)\times(t_o-r^2,t_o]$ (see Theorem~\ref{thm-para}), and a second version on the so-called \emph{intrinsic} cylinders (see below for their definition), which inherit the natural scaling properties of the solution (see Theorem~\ref{thm-intr}). 

The assumption that the right-hand side locally lies in $L^\infty_{\loc}(E_T)$ is not the sharpest possible one: this choice has been done in order to simplify some of the computations to follow, but more general hypotheses could be made.
%%%%%%%%%%%%%%%
\subsection{Novelty and Significance} 
For \emph{elliptic} equations and systems, Meyers \& Elcrat \cite{Meyers:1975} showed that the gradients of solutions locally belong to a slightly higher Sobolev space than expected a priori. The main tools are a reverse H\"older inequality for $\abs{Du}$ and an application of \emph{Gehring's lemma} (see the original paper \cite{Gehring:1973} and also \cite{Giaquinta:1979,Stredulinsky:1980}).
The method works for equations with $p$-growth, hence degenerate and singular elliptic equations of $p$-Laplacian type are allowed.

Giaquinta \& Struwe \cite{Giaquinta:1982} extended the elliptic, local, higher integrability result to parabolic equations. However, in their work, in order to derive the reverse H\"older inequality, the diffusion term $\A$ has {a linear} growth with respect to $\abs{Du}$, so that degenerate and singular equations are ruled out.

The main obstruction to the extension to the degenerate/singular setting is given by the  \emph{lack of homogeneity} in the energy estimates.  This problem can be overcome by using the so-called {\em intrinsic parabolic geometry}, originally developed by DiBe\-ne\-detto \& Friedman~\cite{DiBFri85,DiBenedetto:1993} in the context of the parabolic $p$-Laplace equation. It is a scaling, which depends on the solution itself. Under a more physical point of view, the diffusion process evolves at a time scale which depends instant by instant on $u$ itself; the homogeneity is recovered, once the time variable is rescaled by a factor that depends on the solution in a suitable way.  

Later on, by rephrasing these ideas in the context of intrinsic Calder{\'o}n-Zygmund coverings, Kinnunen \& Lewis \cite{Kinnunen:2000} were able to show that gradients of solutions to equations with the same structure as the parabolic $p$-Laplacian enjoy a higher integrability property, namely
 \begin{align*}
D u \in L^{p+ \eps}_{\loc}(E_T), \qquad \text{for some }\ \eps > 0.
 \end{align*}
This result holds under very general structural assumptions on the operator, and minimal conditions on the right-hand side. The values of $p$ cover the full degenerate range $p > 2$, but are restricted to the \emph{super-critical} singular range $\frac{2n}{n+2} < p < 2$ (this phenomenon is analogous to the above-mentioned restriction on the value of $m$ in \eqref{Eq:m-range}).  Kinnunen \& Lewis's result was then extended to many different contexts: just to mention a few of them, see \cite{Bogelein:2008,Bogelein:2010,Parviainen:2009,Parviainen:2009a}.

Notwithstanding the large amount of results for solutions to $p$-Laplacian-like equations, the higher integrability result 
for the porous medium equations has remained open for quite some years. 

With respect to the $p$-Laplacian situation, the porous medium equation presents a distinctive difficulty: if $c$ is a 
positive
constant, and $u$ is a local solution to \eqref{PMD-eq: model}, then in general $u-c$ is not a solution: this seemingly small 
obstacle makes it impossible to apply to the porous medium equation the approach known for the $p$-Laplacian. Indeed, 
the latter is based on (scaling invariant) Sobolev-Poincar{\'e} inequalities of Gagliardo-Nirenberg type in space-time, 
which are invariant by a constant (i.e. the mean value). 

When dealing with the higher integrability for the gradient of solutions to the porous medium equation, in \cite{GiaSch16} we overcame this difficulty by splitting the problem into two cases: {\em degenerate} and {\em non-degenerate} regimes. This is a very common approach, going back to DeGiorgi. It has been used to get estimates for solutions of PDEs in many different contexts. If we just limit ourselves to solutions to~\eqref{PMD-eq: model}, without pretending to give a full list of all the relevant results, this method is a key tool to derive Harnack inequalities, as well as to prove H\"older continuity: just as an example, see \cite{DiBGiaVes11}. 
However, as far as we can say, its use in the context of gradient estimates for the porous medium equation had not been tried before. 

A second novelty of \cite{GiaSch16} is in Calder{\'o}n-Zygmund covering, where we used cylinders which are intrinsically scaled with respect to what seems the natural quantity here, namely $u^{m-1}$, where $u$ is the solution. 
In other terms, we considered cylinders of the type $Q_{\theta\rho^2,\rho}$ whose space-time scaling is adapted to the
solution $u$ via the coupling 
\[
\dashiint_{Q_{\theta\rho^2,\rho}} u^{m+1}\,dxdt\approx\theta^{-\frac{m+1}{m-1}}.
\]

Finally, there is a third crucial point in the proof of the higher integrability of \cite{GiaSch16}: In the usual approach, as, for example, in \cite{Kinnunen:2000}, one constructs a covering of super-level sets of the spatial gradient with intrinsic cylinders. However, this is not possible for the cylinders which are intrinsically scaled with respect to $u$. Therefore, we weakened this property to so-called {\em sub-intrinsic} cylinders, i.e. cylinders for which the mean value integral of $u^{m+1}$ is not approximately equal to the right-hand side, but only bounded from above by it.

In \cite{GiaSch16} nonnegative solutions to scalar equations are considered; systems have been 
dealt with in \cite{Bogelein:2018}; as a by-product of the vectorial case, B\"ogelein-Duzaar-Korte \& Scheven were are able to consider also signed solutions in the scalar case. Moreover, in \cite{GiaSch16} weak solutions are defined assuming that $Du^{\frac{m+1}2}\in L^2_{\loc}(E_T)$, and as a result they end up satisfying $Du^{\frac{m+1}2}\in L^{2+\epsilon}_{\loc}(E_T)$; in  \cite{Bogelein:2018} the authors start with solutions satisfying $Du^m\in L^2_{\loc}(E_T)$, and as a result, they prove that $Du^m\in L^{2+\epsilon}(E_T)$.

Finally, coming to \eqref{PME}--\eqref{PMS-eq:structure} for $0<m<1$, even though the regularity for \emph{solutions} has been widely studied (see for example \cite{chen-dibe-92,DiBenedetto:1993,Di-Kw-Ve}), much less is known about the properties of the gradient, and, in particular, the higher integrability of $Du^m$ seems to be completely open. The approach we use in these notes, largely follows the same path developed in \cite{GiaSch16}, and described above, the main difference being perhaps in the notion of solution. Indeed, here we assume that $Du^m\in L^2_{\loc}(E_T)$, and this generates some technical difficulties. On the other hand, we avoided any use of the so-called \emph{expansion of positivity} in dealing with the {\em non-degenerate} regime. Therefore, the method could now be employed to study vector-valued problems and signed solutions, but both for simplicity and to strengthen the emphasis on the technical novelties of this work, once more we have just focused on nonnegative solutions to scalar equations. 

The construction of the intrinsic cylinders originally introduced in \cite{Sch13} to study the parabolic $p$-Laplacian for $p\ge2$, is here extended to the singular setting, and the proof is given with all details; indeed, we think it is of independent interest, and it might find applications in other contexts. 

In order to underline the potential use of the intrinsic cylinders built in Lemma~\ref{lem:scal}, we show that they can be used to provide bounds on the solutions in a new, clean formulation (see Proposition~\ref{Prop:B:4:1}), even though qualitative and quantitative bounds of solutions to the singular porous medium equations are well-known (see, for example, \cite[Appendix B]{DiBGiaVes11}).
%%%%%%%%%%%%%%%%%
\vskip.2truecm
\noindent{\it Acknowledgements.}
S.~Schwarzacher is a member of the Ne\v{c}as Centre for Mathematical Modeling. 
%%%%%%%%%%%%%%
\section{Preliminaries}

\subsection{Notation}
Consider a point $z_{\origin} = (x_{\origin},t_{\origin}) \in \R^{n+1}$
and two parameters $\rho, \tau > 0$. The open ball with radius $\rho$ and center $x_{\origin}$ will be denoted by
 \begin{align*}
B_{\rho}(x_{\origin}) \deq \set{x \in E \st \abs{x-x_{\origin}} < \rho},
 \end{align*}
whereas the time interval with width  $\tau$ and center $t_o$ will be denoted by
\[
\Lambda_{\tau}(t_{\origin})\deq \set{t \in (0,T] \st t_o-\tau<t<t_o+\tau}.
\]
Finally, we define the time-space cylinder by
 \begin{align*}
Q_{\tau,\rho}(z_{\origin})  \deq \pa{t_{\origin} - \tau, t_{\origin} + \tau}\times B_{\rho}(x_{\origin}) .
 \end{align*}
As we prove local estimates, the reference point is never of importance, and we often omit it by writing $B_{\rho}$ and $Q_{\tau,\rho}$.
%
%Given an integrable function $u$, we denote its average in the space variable over a ball $B_{\rho}(x_{\origin})$ by
 %\begin{align*}
%\mean{} \deq  \vint_{B_{\rho}(x_{\origin})} u(x,t)\, dx = \frac{1}{\abs{B_{\rho}}} \int_{B_{\rho}(x_{\origin})} u(x,t)\, dx;
 %\end{align*}
%the space-time average over the cylinder $Q_{\tau}(z_{\origin})$ is denoted by
 %\begin{align*}
%a(Q_{\rho, \tau}) \deq \viint_{Q_{\rho, \tau}(z_{\origin})} u(x,t)\,dx\,dt = \frac{1}{\abs{Q_{\rho, \tau}}} \iint_{Q_{\rho, \tau}(z_{\origin})} u(x,t)\,dx\,dt.
 %\end{align*}

%Given an integrable function $u$, we denote its average in the space variable over a ball $B_{\rho}(x_{\origin})$ by
 %\begin{align*}
%\mean{} \deq  \vint_{B_{\rho}(x_{\origin})} u(x,t)\, dx = \frac{1}{\abs{B_{\rho}}} \int_{B_{\rho}(x_{\origin})} u(x,t)\, dx;
 %\end{align*}
%the space-time average over the cylinder $Q_{\tau}(z_{\origin})$ is denoted by
 %\begin{align*}
%a(Q_{\rho, \tau}) \deq \viint_{Q_{\rho, \tau}(z_{\origin})} u(x,t)\,dx\,dt = \frac{1}{\abs{Q_{\rho, \tau}}} \iint_{Q_{\rho, \tau}(z_{\origin})} u(x,t)\,dx\,dt.
 %\end{align*}
The symbol $\abs{\spacedot}$ stands for the Lebesgue measure, either in $\R^{n}$ or $\R^{n+1}$, and the dimension will be clear from the context.

Moreover, we write $f \sim g$ if there exist constants $c,\,C >0$, which depend only on the data, such that $cf\le g\le Cf$.

\subsection{An auxiliary result}
\begin{proposition}
Given $u, a \geq 0$ and $0 < m < 1$, we have
\begin{align}
\frac{1}{2} \pa{u-a}\pa{u^{m} - a^{m}} &\leq  \int_{a}^{u} \pa{y^{m} - a^{m}}\,dy \leq \pa{u-a} \pa{u^{m} - a^{m}} \qquad \text{if $\ u \geq a$},\label{PMS-eq: bound integral +}\\
\frac{m}{2} \pa{a-u}\pa{a^{m} - u^{m}} &\leq \int_{u}^{a} \pa{a^{m} - y^{m}}\,dy \leq \pa{a-u}\pa{a^{m} -u^{m}}\qquad \text{if $\ u < a$}.\label{PMS-eq: bound integral -}
\end{align}
\end{proposition}

\begin{proof}
The function $f(y) = y^{m} - a^{m}$ is concave, so the integral in \eqref{PMS-eq: bound integral +} is bounded from below by the area of the triangle, and from above by the area of the rectangle.

The integral in \eqref{PMS-eq: bound integral -} is estimated
from above  using the area of the rectangle once more; using the triangle, the bound form below follows from the inequality
\[
\frac{1}{2}m  \sq{a^{m-1} (a-u)} (a-u) \geq \frac{1}{2}m \pa{a^{m} - u^{m}} \pa{a-u},
\]
since $0< m < 1$ and $u < a$.
\end{proof}

\subsection{Constants and data}
As usual, the letter $c$ is reserved to positive constants, whose value may change from line to line, or even in the same formula.
We say that a generic constant $c$ \emph{depends on the data}, if $c = c(n, m, \nu , L)$,
where $\nu$ and $L$ are the quantities introduced in \eqref{PMS-eq:structure}.

Let $\eta\in L^\infty(E)$ and nonnegative: we let 
 \begin{align*}
(g)_E^\eta=\frac{1}{\norm{\eta}_{L^1(E)}}\int_E g\,\eta\,dx.
 \end{align*}
In the special case of {$\eta\equiv1$}, we write
 \begin{align*}
(g)_E^1=:(g)_E=: \dashint_Eg\,dx.
 \end{align*}
We will frequently use what we will refer to in the following as \emph{the best constant property}. For positive $\eta$ we have, for any $c\in\setR$ and $q\in [1,\infty)$
\begin{align}\label{bcp}
\bigg(\frac{1}{\norm{\eta}_{L^1(E)}}\int_E\abs{g-(g)_E^\eta}^q\eta\,dx\bigg)^\frac1q\leq 2\bigg(\frac{1}{\norm{\eta}_{L^1(E)}}\int_E\abs{g-c}^q\eta\,dx\bigg)^\frac1q.
\end{align}
Moreover, if $0\leq \eta\leq 1$, by \eqref{bcp} one obviously obtains that
\begin{align}\label{meanchange}
\bigg(\frac{1}{\norm{\eta}_{L^1(E)}}\int_E\abs{g-(g)_E^\eta}^q\eta\,dx\bigg)^\frac1q\leq 2\bigg(\frac{1}{\norm{\eta}_{L^1(E)}}\int_E\abs{g-(g)_E}^q dx\bigg)^\frac1q,
\end{align}
and as a consequence, that
\begin{equation}\label{meanchange2}
\begin{aligned}
\abs{(g)_E^\eta-(g)_E}&\leq 
\bigg(\frac{1}{\norm{\eta}_{L^1(E)}}\int_E\abs{g-(g)_E^\eta}^q\eta\,dx\bigg)^\frac1q\\
&\leq 2\bigg(\frac{1}{\norm{\eta}_{L^1(E)}}\int_E\abs{g-(g)_E}^q dx\bigg)^\frac1q. 
\end{aligned}
\end{equation}
We will also use the following estimate that was first proved in \cite[Lemma~A.2]{DieKapSch11}.
\begin{lemma}\label{Lm:2.2}
For $q\in (\frac12,\infty)$ we have 
\begin{align}
\label{trick}
\dashint_E\abs{g^q-\Mean{g}_E^q}^2 dx
\leq c_o\dashint_E\abs{g^q-(g^q)_E}^2 dx
\leq c_1\dashint_E\abs{g^q-(g)_E^q}^2 dx
\end{align}
{where $c_o$ and $c_1$ are constants that depend only on the data.}
The same estimate holds for $q\in (0,\frac12]$ in case
\[
\sup_E \abs{g}\leq K\Mean{g}_E
\]
for some $K>0$.
\end{lemma}
\begin{proof} By the Fundamental Theorem of Calculus, and the orthogonality of the mean value to constants we have
\begin{align*}
\dashint_E\abs{g^q-\Mean{g}_E^q}^2 dx
&\sim \dashint_E(g^{2q-1}-\Mean{g}_E^{2q-1})\cdot(g-\Mean{g}_E) dx\\
&= \dashint_E(g^{2q-1}-\Mean{g^q}_E^\frac{2q-1}q)\cdot(g-\Mean{g}_E) dx.
\end{align*}
In case that $q>\frac12$, we find that 
\begin{align*}
\abs{g^{2q-1}-\Mean{g^q}_E^\frac{2q-1}q}\sim (g+\Mean{g^q}^\frac1q)^{2q-2}\abs{g-\Mean{g^q}^\frac1q},
\end{align*}
and conclude by dividing in the cases $\abs{u-\Mean{g}_E}\leq c \abs{g-(g^q)_E^\frac1q}$ and its opposite.

In case that $q\in (0,\frac12)$, we find 
\begin{align*}
\dashint_E\abs{g^q-\Mean{g}_E^q}^2 dx
&\sim \dashint_E(g+\Mean{g}_E)^{2q-2}\abs{g-\Mean{g}_E}^2 dx
\\
&\leq \Mean{g}_E^{2q-2}\dashint_E\abs{g-\Mean{g}_E}^2 dx
\\
&\leq \Mean{g}_E^{2q-2}\dashint_E\abs{g-\Mean{g^q}_E^\frac1q}^2 dx
\\
&\leq c_K\dashint_E(g+\Mean{g}_E)^{2q-2}\abs{g-\Mean{g^q}_E^\frac1q}^2 dx
\\
&\leq c_K\dashint_E(g+\Mean{g^q}_E^\frac1q)^{2q-2}\abs{g-\Mean{g^q}_E^q}^2 dx
\sim \dashint_E\abs{g^q-\Mean{g^q}_E}^2 dx
\end{align*}

\end{proof}

The restriction on $q>\frac12$ in Lemma~is rather strong.
%the above  consequences. Indeed, in the case $m\in(0,\frac12]$ changes in the mean (namely, 
%a change from space-time means to means done only with respect to space) are not possible for the natural 
%quantity $u^m$, and have to be somehow artificially shifted to $u^\frac{m+1}{2}$ 
%(since $\frac{m+1}2>\frac12$ always.) This is particularly important for estimates of unbounded solutions in 
%the degenerate case (see below for more details). 
However, surprisingly enough, a different argument allows 
to circumvent the problems of small $m$. As the following lemma shows, for any convex quantity a ``mean change" 
is possible.
\begin{lemma}\label{lem:trick}
Let $\eta\in L^\infty(E)$ and nonnegative. For $m\in (0,1)$, $p\geq \frac{1}{m}$, and $u\in L^p(E)$ and nonnegative, we have 
\begin{align*}
\bigg(\frac{1}{\norm{\eta}_1}\int_E\abs{u^m-(\Mean{u}_E^\eta)^m}^p \eta dx\bigg)^\frac1p
&\leq c_o\bigg(\frac{1}{\norm{\eta}_1}\int_E\abs{u^m-(u^m)_E^\eta}^p \eta dx\bigg)^\frac{1}{p}
\\
&\leq 2c_o\bigg(\frac{1}{\norm{\eta}_1}\int_E\abs{u^m-(\Mean{u}_E^\eta)^m}^p \eta dx\bigg)^\frac1p,
\end{align*}
where $c_o$ is a constant that depends only on $m$ and $p$.
\end{lemma}
\begin{proof}
The first estimate follows by the best constant property. Concerning the second estimate, let
\[
\lambda=\Mean{u}_E^\eta\quad \text{ and }\quad e^m=(u^m)_E^\eta.
%,a=(u^{m+1})_E^\frac1{m+1}
%a=\frac{1}{\abs{\set{u\leq k e}\cap B}}\int_{E}\chi_{\set{u\leq k e}} g\, dx.
\]
By Jensen's inequality $e\leq \lambda$.
Now, for any $a\in \setR_+$  we estimate
\begin{align*}
\bigg(\frac{1}{\norm{\eta}_1}\int_E \abs{u^m-\lambda^m}^p\eta\,dx\bigg)^\frac1{p}\leq c \bigg(\frac{1}{\norm{\eta}_1}\int_E \abs{u^m-a^m}^p\eta\,dx\bigg)^\frac1p
+c \abs{a^m-\lambda^m}.
\end{align*}
Since for $x\in \set{u> 2a}$
\[
u^m(x)\leq \abs{u^m(x)-a^m}+a^m\leq \abs{u^m(x)-a^m}+\frac{u^m(x)}{2^m},
\]
we find
\[
u^m\chi_\set{u> 2a}\leq \frac{2^m}{2^m-1}\abs{u^m-a^m}\chi_\set{u> 2a}.
\]
Now 
\begin{align*}
\abs{a^m-\lambda^m}\sim (a+\lambda)^{m-1}\abs{a-\lambda}\leq& (a+\lambda)^{m-1}\frac{1}{\norm{\eta}_1}\int_{E}\abs{u-a}\eta\, dx\\
\leq& (a+\lambda)^{m-1}\frac{1}{\norm{\eta}_1}\int_{E}\abs{u-a}\chi_{\set{u\leq 2a}}\eta\, dx\\
&+ (a+\lambda)^{m-1}\frac{1}{\norm{\eta}_1}\int_{E}\abs{u-a}\chi_{\set{u> 2a}}\eta\, dx\\
\leq& \frac{a^{m-1}}{\norm{\eta}_1}\int_{E}\abs{u-a}\chi_{\set{u\leq 2a}}\eta\, dx\\
&+ \frac{\lambda^{m-1}}{\norm{\eta}_1}\int_{E}u\chi_{\set{u> 2a}}\eta \, dx\\
\leq& \frac{c}{\norm{\eta}_1}\int_{E}(a+u)^{m-1}\abs{u-a}\chi_{\set{u\leq 2a}}\eta \, dx\\
&+\bigg(\frac{1}{\norm{\eta}_1}\int_{E}u\chi_{\set{u> 2a}}\eta\, dx\bigg)^m\\
\leq& \frac{c}{\norm{\eta}_1}\int_{E}\abs{u^m-a^m}\chi_{\set{u\leq 2a}}\eta\, dx\\
&+c\bigg(\frac{1}{\norm{\eta}_1}\int_{E}\abs{u^m-a^m}^\frac{1}{m}\chi_{\set{u> 2a}}\eta\, dx\bigg)^{m},
\end{align*}
and this implies the result by Jensen's inequality.
\end{proof}
 We close the section with the following estimate, which was originally derived in \cite[Lemma~A.1]{Sch13}. See also~\cite{GiaSch16}. 

\begin{lemma}\label{seb-app}
Let $Q_1\subset Q$ be two cylinders and $f\in L^q(Q)$ for some $q\in[1,\infty)$. If for some $\epsilon\in(0,1)$ we have
\[
|(f)_{Q_1}|\le\epsilon [(|f|^q)_Q]^{\frac1q},
\] 
then
\[
|(f)_{Q_1}|\le\epsilon [(|f|^q)_Q]^{\frac1q}\le\frac{\epsilon}{1-\epsilon}\left(1+\left(\frac{|Q|}{|Q_1|}\right)^{\frac1q}\right)\left(\dashint_Q |f-(f)_Q|^q\,dxdt\right)^{\frac1q}.
\]
\end{lemma}
%%%%%%%%%%%%%%%%%%%%%%%%%%%%%%%%%%%
\section{Constructing Proper Cylinders}
Since the construction of intrinsic cylinders below has the potential to be significant for future application not only for the porous medium equation, but also for general singular systems of various kinds, we will use a notation that emphasizes the generality of the approach. Indeed, the covering can be introduced with respect to any integrable function $f$, related to a scaling exponent $p$ or $m$. 

Consider a degenerate or singular diffusion equation or system: by \emph{intrinsic} we mean a scaling that inherits the local ellipticity coefficient (i.e. the diffusion coefficient) with respect to the mean value of the function. For the porous medium equation, the ellipticity coefficient is of order $m u^{m-1}$. In this case we will call $Q_{s,\sqrt{s/\theta}}$ a \emph{$K$-intrinsic cylinder}, if 
\begin{align}\label{def:intr0}
\frac{\theta}{K}\leq \Mean{u^{m+1}}_{Q_{s,\sqrt{s/\theta}}}^\frac{1-m}{m+1}\leq {K}{\theta},
\end{align}
for some $K\ge1$. We will call $Q_{s,\sqrt{s/\theta}}$ a \emph{$K$-sub-intrinsic cylinder}, if only the estimate from above holds in \eqref{def:intr0}.  Observe, that (in the intrinsic case) $\theta^{-1}\sim u^{m-1}$ is somehow proportional to the scaling of ellipticity.

In the case of the $p$--Laplacian
\begin{equation}\label{plap}
\begin{aligned}
&u \in C^{0}_{\loc}\pa{0,T; L^{2}_{\loc}(E)} \cap L^{p}_{\loc}\pa{0,T; W^{1,p}_{\loc}(E)}\\
&u_{t} - \div\pa{\abs{Du}^{p-2} Du} = f\qquad \text{weakly in $E_T$,}
\end{aligned}
\qquad p > 1,
\end{equation}
 the ellipticity coefficient is of order $\abs{D u}^{p-2}$. Here one usually considers the size $\lambda\sim \abs{D u}$. The canonic analogue of $\theta$ is therefore $\lambda^{2-p}$, and $p-1$ corresponds to $m$.  

Observe that this type of cylinders were initially introduced for the degenerate $p$--Laplacian (cf. \cite{Sch13}). For this reason alone, we will stick to the exponent $p\equiv m+1$ and we will relate it to a general integrable function $f$; one only needs to substitute $f$ with $\abs{D u}^p$ or $u^{m+1}$, or any other suitable function, depending on the problem under consideration. 

In terms of the function $f$ and the exponent $p$, we will then call $Q_{s,\sqrt{s/\theta}}(z)=:Q_{s,\sqrt{s\lambda^{p-2}}}(z)$ a \emph{$K$-intrinsic cylinder}, if 
\begin{align}\label{def:intr}
\frac{\theta}{K}=\frac{\lambda^{2-p}}{K}\leq \Mean{\abs{f}}_{Q_{s,\sqrt{\lambda^{p-2}s}}}^\frac{2-p}{p}\leq K\lambda^{2-p}=K\theta,
\end{align}
for some $K\ge1$, and we will call $Q_{s,\sqrt{s\lambda^{p-2}}}$ a \emph{$K$-sub-intrinsic cylinder}, if only the estimate from above in \eqref{def:intr} holds true.

In the following, we will avoid any reference to the constant $K$, meaning that the cylinders will be either intrinsic or sub-intrinsic for some proper $K$. Moreover, $Q_s^{\theta_s}$, $Q_s^{\lambda_s}$, and $Q_{s,r(s)}$ will all denote the same object: the use of either one of the three possible notation depends on the particular geometric feature we want to emphasize.
\begin{lemma}\label{lem:scal}
Let $p\in (\frac{2n}{n+2}, 2)$. Let $Q_{S,R}(t,x)\subset \setR^{n+1}$, $f\in L^1(Q_{S,R}(t,x))$ and $\hat{b}\in (0,1/2]$, such that $\hat{b}<(n+2)p-2n$. For every $0<s\leq S$ there exist $r(s)$, $\lambda_s$, and $Q_{s,r(s)}(t,x)$ with the following properties. Let $s,\sigma\in (0,R]$ and $s<\sigma$, then 
\begin{enumerate}[(a)]
\item\label{scal:-1}  $0\leq r(s)\leq R$ and $r(s)=:\sqrt{\lambda_s^{p-2}s}=:\sqrt{s/\theta_s}$. In particular, $Q_{s,r(s)}(t,x)=:Q_s^{\lambda_s}=:Q_s^{\theta_s}\subset E_T$.
\item\label{scal:0}  $r(s)\leq \big(\frac{s}{\sigma}\big)^{\hat{b}}r(\sigma)$, the function $r=r(s)$ is continuous and strictly increasing on $[0,S]$. In particular, $Q_s^{\lambda_s}\subset Q_\sigma^{\lambda_{\sigma}}$.
 \item \label{scal:2} $\displaystyle\dashiint_{Q_s^{\lambda_s}}\abs{f}\,dxdt\leq\lambda_s^p$, i.e. $Q_s^{\lambda_s}$ is sub-intrinsic. %in case $f=u^{m+1}$, $p=m+1$. 
\item \label{scal:5}  If $r(s)<\big(\frac{s}{\sigma}\big)^{\hat{b}}r(\sigma)$, then there exists $s_1\in[s,\sigma)$ such that $Q^{\lambda_{s_1}}_{s_1}$ is intrinsic.
\item \label{scal:6} If for all $s\in (s_1,\sigma)$, $Q_s^{\lambda_s}$ is strictly sub-intrinsic, then $\lambda_{s}^{2-p}\leq\big(\frac{s}{\sigma}\big)^{\beta}\lambda_{\sigma}^{2-p}$ for all $s\in [s_1,\sigma]$ and $\beta=(1-2\hat{b})\in [0,1)$.
 \item \label{scal:3} For $\gamma\in (0,1]$, we have
 \begin{align*}
 &r(s)\leq c\gamma^{-\hat{a}}r(\gamma s),\\ 
 &\abs{Q_{\gamma s}^{\lambda_{\gamma s}}}^{-1}
\leq c\gamma^{-n\hat{a}-2}\abs{{Q_{s}^{\lambda_{s}}}}^{-1},\\ 
&\lambda_{\gamma s}^{2-p}\leq c\gamma^{-\max\set{2\hat{a},\beta}}\lambda_s^{2-p},
\end{align*} 
with $\hat{a}=\hat b+\frac{2}{2p-(2-p)n}$. Observe, that $\hat{a}\to \infty$, as $p\searrow\frac{2n}{n+2}$.
\item \label{scal:7} %For $\gamma\in(0,1]$, we find $Q_{s,r(s)}\subset  Q_{\gamma\sigma,r(\gamma \sigma)}$ for $s=\gamma^{\hat{b}} \sigma$, with $\hat{b}$. \textcolor{red}{\bf I do not think this is correct, since $s=\gamma^{\hat b}\sigma$ is larger than $\gamma\sigma$, whenever $\hat b\in(0,1)$}. 
For $c>1$ we have $Q_{cs,cr(s)}\subset Q_{\tilde{c} s,r(\tilde{c}s)}\subset Q_{\overline{c} s,\overline{c} r(s)}$ for $\tilde c=c^\frac1{\hat b}$ and $\overline{c}=\max\set{\tilde c,\tilde{c}^{\hat{a}}}$.
\end{enumerate}
The constant $c$ only depends on the dimension $n$ and on $p$.
\end{lemma}

\begin{proof}
Let $Q_{S,R}(t,x)\subset E_T$. In the following we often omit the point $(t,x)$. Moreover, whenever we integrate both in space and in time, we use the symbol $dz$.
We start by defining for every $s\in (0,S]$ the quantity
\begin{align*}
\tilde{r}(s)
=\sup\Bigset{\rho<R\,\Big|\,
\bigg(\int^{t+\frac{s}{2}}_{t-\frac{s}{2}}\int_{B_{\rho}(x)}\abs{f} dz\bigg)^{2-p}\rho^{2p}\abs{B_\rho}^{p-2}\leq s^2}.
\end{align*}

The function $\tilde{r}(s)$ is well defined and strictly positive for $r>0$. This is due to the fact that
\[
b_0:=(p-2)n+2p=(n+2)p-2n>0,
\]
since we assume $\displaystyle p\in \Big(\frac{2n}{n+2},2\Big)$.We define $\tilde{\lambda}_s$ by the equation $\tilde{\lambda}_s^{2-p}=\frac{s}{\tilde{r}^2(s)}$.
By construction we find that
\begin{align}
  \label{eq:sforp}
\begin{aligned}
  \bigg(\dashiint_{Q_{{s,\tilde{r}(s)}}}\abs{f} dz\bigg)^{2-p} {\tilde r(s)}^{2p}\leq s^{p}.
% \intertext{For $p>2$ we can rewrite this as}
%   s^\frac{2}{p-2}\int_{-s+t}^{t+s}
%   \int_{B_r}\abs{D u}^p \dz=r^\frac{2p+np-2n}{p-2}.
\end{aligned}
\end{align}
This implies that
 \[
 \bigg(\dashiint_{Q_{{s,\tilde{r}(s)}}}\abs{f} dz\bigg)^{2-p} \leq \Big(\frac{s}{{\tilde r}(s)^2}\Big)^{p},
\]
that is
\begin{align}
 \label{eq:subint}
\dashiint_{Q_{{s,\tilde{r}(s)}}}\abs{f} dz \leq \tilde{\lambda}_s^p.
 \end{align}
Next, we will show that $\tilde{r}(s)$ is continuous for $s\in(0,S]$. First of all, it is enough to show that $\kappa(s):=\tilde r(s)^\frac{2p}{2-p}\abs{B_{\tilde r(s)}}^{-1}$ is continuous. 

Take $0<s_1<s< S$. For $\epsilon>0$, there exists a $\delta>0$, such that
\[
\int^{t-s_1}_{t-s}\int_{B_{R}}\abs{f} dz\leq \epsilon\qquad \text{ for all }\ {s-s_1<\delta}.
\]
Now, in case $\kappa(s_1)<\kappa(s)(\leq R^\frac{2p}{2-p}\abs{B_{R}}^{-1})$, we find that 
\begin{align*}
{\kappa(s)}\,\int^{t}_{t-s}\int_{B_{\tilde r(s)}}\abs{f} dz-{\kappa(s_1)}\,\int^{t}_{t-s_1}\int_{B_{{\tilde r}(s_1)}}\abs{f} dz\leq s^\frac{2}{2-p}-s_1^\frac{2}{2-p}\leq c\delta^\frac{2}{2-p}
\end{align*}
and hence
\begin{align*}
\left[\int^{t}_{t-s_1}\int_{B_{\tilde r(s_1)}}\abs{f} dz\right]\abs{\kappa(s)-\kappa(s_1)}\leq c\delta^\frac{2}{2-p},
\end{align*}
which concludes this argument, since $\displaystyle\int^{t}_{t-s_1}\int_{B_{\tilde r(s_1)}}\abs{f} dz>0$ when $\tilde r(s_1)<R$.

\noindent In case $\kappa(s)<\kappa(s_1)(\leq R^\frac{2p}{2-p}\abs{B_{R}}^{-1})$, we find
\begin{align*}
\kappa(s)\,\int^{t}_{t-s_1}\int_{B_{\tilde r(s_1)}}\abs{f} dz+\epsilon\geq {\kappa(s)}\,\int^{t}_{t-s}\int_{B_{\tilde r(s)}}\abs{f} dz\geq s^\frac{2}{2-p},
\end{align*}
since $\int^{t}_{t-s_1}\int_{B_{\tilde r(s_1)}}\abs{f} dz>0$.
Hence
\begin{align*}
\int^{t}_{t-s_1}\int_{B_{\tilde r(s_1)}}\abs{f} dz\abs{\kappa(s_1)-\kappa(s)}\leq {s_1^\frac{2}{2-p}-s^\frac{2}{2-p}}%c\delta^\frac{2}{2-p}
+\epsilon\leq \epsilon.
\end{align*}
This concludes the proof of the continuity, as the case $\tilde r(s)=\tilde r(s_1)$ is trivial.

Now it might happen, that $s<\sigma$ and $\tilde{r}(s)>\tilde{r}(\sigma)$. To avoid that, for $\hat{b}\in (0,b_0)$ we define 
\begin{align*}
r(s)={\min_{s\leq a\leq S}}\Big(\frac{s}{a}\Big)^{\hat{b}}\tilde{r}(a). 
\end{align*}
The minimum exists, as $\Big(\frac{s}{a}\Big)^{\hat{b}}\tilde{r}(a)$ is continuous in $a$.
If we take $\sigma\in(s,S]$
\begin{align}
\label{eq:grow}
 r(s)=\min\Bigset{{\min_{s\leq a\leq \sigma}}\Big(\frac{s}{a}\Big)^{\hat{b}}\tilde{r}(a),\Big(\frac{s}{\sigma}\Big)^{\hat{b}}r(\sigma)}
\end{align}
we find that
$r(s)< r(\sigma)$. Now we define $\lambda_s:=\big(\frac{r(s)^2}{s}\big)^\frac1{p-2}\geq \tilde{\lambda}_s$ and $Q_s^{\lambda_s}:=Q_{s,r(s)}$. By this definition we find \eqref{scal:-1} and \eqref{scal:0}, as $\lim_{s\to0}r(s)\leq \lim_{s\to 0}\big(\frac{s}{S}\big)^{\hat{b}}{R}=0$. 

We show \eqref{scal:2}, by \eqref{eq:sforp}--\eqref{eq:subint} 
\begin{equation}\label{eq:mon}
\begin{aligned}
 \dashiint_{Q_{s,r(s)}}\abs{f}\, dz 
 \leq \frac{\abs{B_{\tilde{r}(s)}}}{\abs{B_{r(s)}}} \dashiint_{Q_{s,\tilde{r}(s)}}\abs{f}\, dz
&=
\Big(\frac{{\lambda}_s}{\tilde{\lambda}_s}\Big)^\frac{(2-p)n}{2}\dashiint_{{Q_{s,\tilde{r}(s)}}}\abs{f}\, dz\\
&\leq \tilde{\lambda}_s^\frac{2p+np-2n}{2}{\lambda}_s^\frac{2n-pn}{2}\leq \lambda_{s}^p.
\end{aligned}
\end{equation}
To prove \eqref{scal:5} we assume that $r(s)<\big(\frac{s}{\sigma}\big)^{{\hat b}}r(\sigma)$. Then there exists $s_1\in [s,\sigma)$, such that 
\[
\Big(\frac{s}{s_1}\Big)^{\hat{b}}\tilde{r}(s_1)=r(s)=\min_{s\leq a\leq S}\Big(\frac{s}{a}\Big)^{\hat{b}}\tilde{r}(a)\leq \Big(\frac{s}{s_1}\Big)^{\hat{b}}\min_{s_1\leq a\leq S}\Big(\frac{s_1}{a}\Big)^{\hat{b}}\tilde{r}(a)=\Big(\frac{s}{s_1}\Big)^{\hat{b}} r(s_1).
\]
This implies that $\tilde{r}(s_1)=r(s_1)$. Since by our assumption we also have that
\[
r(s_1)<\big(\frac{s_1}{\sigma}\big)^{\hat{b}}r(\sigma)\leq \big(\frac{s_1}{S}\big)^{\hat{b}} R,
\] 
we deduce by \eqref{eq:sforp} that $Q_{s_1,r(s_1)}=Q_{s_1}^{\lambda_{s_1}}$ is intrinsic. This implies \eqref{scal:5}. 

By \eqref{scal:5}, if $Q_a^{\lambda_a}$ is strictly sub-intrinsic for all $a\in(s,\sigma)$, then 
$r(a)=\big(\frac{a}{\sigma}\big)^{\hat{b}}r(\sigma)$ for all $a\in(s,\sigma)$.
Now we {compute}
\[
 \lambda_a^{2-p}= \frac{a}{r(a)^2}=\frac{a}{\big(\frac{a}{\sigma}\big)^{2\hat{b}}r(\sigma)^2}=\Big(\frac{a}{\sigma}\Big)^{1-2\hat{b}}\lambda_\sigma^{2-p},
\]
and this yields \eqref{scal:6}, with $\beta={1-2\hat{b}}$.

To prove \eqref{scal:3} we take $\gamma\in(0,1)$. First we show the estimate for $r(s)$. If $r(\gamma s)=\gamma^{\hat{b}}r(s)$ there is nothing to show. If $r(\gamma s)<\gamma^{\hat{b}}r(s)$, we find by \eqref{scal:5} that there is a $\gamma_1\in[\gamma,1)$ with $r(\gamma s)=\big(\frac{\gamma}{\gamma_1}\big)^{\hat{b}} r(\gamma_1 s)$ and $Q_{\gamma_1 s}^{\lambda_{\gamma_1 s}}$ is intrinsic. 
Next, by the sub-intrinsicity of $Q^{\lambda_s}_s$ and the intrinsicity of $Q^{\lambda_{\gamma_1s}}_{\gamma_1 s}$ and by the calculation made in \eqref{eq:subint} and \eqref{eq:sforp}, we find that
\begin{align*}
&\bigg({\iint_{Q_{s,r(s)}}}\abs{f} dz\bigg)^{2-p}(r(s))^{2p-(2-p)n}\abs{B_1}^{2p-(2-p)n}
\\&\quad \leq s^2=\frac{\gamma_1^2s^2}{\gamma_1^2}
 =\frac{1}{\gamma_1^2} \bigg({\iint_{Q_{\gamma_1 s,r(\gamma_1 s)}}}\abs{f} dz\bigg)^{2-p}(r(\gamma_1 s))^{2p-(2-p)n}\abs{B_1}^{2p-(2-p)n}.
\end{align*}
%\textcolor{red}{\bf [I deleted $\abs{B_1}^{2p-(2-p)n}$, which is 1 anyway.]}
Since by construction $Q^{\lambda_{\gamma_1s}}_{\gamma_1 s}\subset Q^{\lambda_s}_s$, this implies that
\[
(r(s))^{2p-(2-p)n}\leq \frac{1}{\gamma_1^2}(r(\gamma_1 s))^{2p-(2-p)n}
\]
and so
\[
r(s)\leq \frac{c}{\gamma_1^\frac{2}{2p-(2-p)n}}r(\gamma_1 s)=\frac{c}{\gamma_1^\frac{2}{2p-(2-p)n}}\Big(\frac{\gamma_1}{\gamma}\Big)^{\hat{b}}r(\gamma s)\leq\frac{c}{\gamma^{\hat b+\frac{2}{2p-(2-p)n}}}r(\gamma s)=c \gamma^{-\hat{a}}r(\gamma s),
\]
with $\hat{a}=\hat b+\frac{2}{2p-(2-p)n}$. Hence the estimate on $r(s)$ is proved. 

The estimate on the size of the intrinsic cylinders is a direct consequence. In order to give an estimate for the $\lambda_s$, {once more} we may assume that $r(\gamma s)<{\gamma^{\hat b}}r(s)$, since otherwise $\lambda_{\gamma s}^{2-p}=\gamma^\beta\lambda_s^{2-p}$ and there is nothing to show. Hence, by the above definitions and \eqref{scal:2} we find
\begin{align*}
%\label{eq:mon2}
 \lambda_{\gamma s}^{2-p}= \Big(\frac{\gamma}{\gamma_1}\Big)^{\beta }\lambda_{\gamma_1 s}^{2-p}=\Big(\frac{\gamma}{\gamma_1}\Big)^{\beta}\frac{\gamma_1 s}{(r(\gamma_1 s))^2}\leq \Big(\frac{\gamma}{\gamma_1}\Big)^{\beta }\Big(\frac{r(s)}{r(\gamma_1 s)}\Big)^2\lambda_s^{2-p}
 \leq c\gamma^{-2\hat{a} }\lambda_s^{2-p}
 %=\frac{c}{r(\gamma_1 s)^n \gamma_1 s}\int\limits_{Q_{\gamma_1 s,r}}\abs{f}^p\dz
% \leq\frac{c}{\gamma_1} \Big(\frac{r(s)}{r(\gamma_1 s)}\Big)^n
%  \dashint_{Q_{s,r(s)}}\abs{f}^p\dz
% \leq \frac{c\lambda_s^p}{\gamma}.            
\end{align*}
which closes the argument for \eqref{scal:3}. Finally, \eqref{scal:7} follows directly by \eqref{scal:0} and \eqref{scal:3}.
\end{proof}
In the sequel, we will make heavy use of the fact that for $\gamma\in (0,1]$, up to a constant,
\begin{equation}\label{eq:r}
\begin{aligned}
r(\gamma s)&\leq \gamma^{\hat{b}} r(s)\leq \gamma^{\hat{b}-\hat{a}}r(\gamma s),
\\
\theta_{s}&\leq \gamma^{2\hat{b}-1}\theta_{s\gamma}\leq \gamma^{2(\hat{b}-\hat{a})}\theta_{s}.
\end{aligned}
\end{equation}

We {will} also use the following version of Vitali's covering from \cite[Lemma 5.4]{GiaSch16}.
It is inspired by \cite[Chapter~1, Lemma~1 and Lemma~2]{stein:1993}. See also \cite[Paragraph~1.5, Theorem~1]{Ev-Gar}.
%%%%%%%%%%%%%%%%%%%%%%%%%%%%
\begin{lemma}\label{Lm:Vitali:1}
Let $\Omega\subset \setR^M$ and $R\in \setR$. Let there be given a two-parameter family ${\mathcal F}$ of nonempty and open sets
\[
\set{U(x,r)\,|\,x\in \Omega,\, r\in (0,R]},
\] 
which satisfy 
\begin{enumerate}[(i)]
\item They are \emph{nested}, that is, 
\begin{equation}\label{Eq:nested}
\text{ for any }\ x\in\Omega,\ \text{ and }\ 0<s<r\le R,\ U(x,s)\subset U(x,r);
\end{equation}
\item There exists a constant $c_1{>1}$, such that 
\begin{equation}\label{vit-2}
U(x,r)\cap U(y,r)\neq \emptyset\ \ \Rightarrow\ \ U(x,r)\subset U(y,c_1r).
\end{equation}
\item There exists a constant $a>1$ such that, for all $r\in (0,R]$,
\begin{equation}\label{vit-1}
0<\abs{U(x,2r)}\leq a\abs{U(x,r)}<\infty.
\end{equation}
\end{enumerate}
Then we can find a disjoint subfamily 
$\set{U_i}_{i\in \setN}=\set{U(x_i,r_{x_i})}_{i\in\setN}$, such that
\[
 \bigcup\limits_{x\in\Omega}U(x,r_x)\subset \bigcup_{i\in \setN} \tilde{U}_i,
 \]
 with $\tilde{U}_i=U(x_i,2c_1 r_{x_i})$, $\abs{U_i}\sim \abs{\tilde{U}_i}$
 and
 \[
\abs{\Omega}\leq c\sum_i\abs{U_i},
 \]
where the constant $c>1$ depends only on $c_1$, $a$, and the dimension $M$.
\end{lemma}
Notice that the family of cylinders built in Lemma~\ref{lem:scal} satisfies \eqref{Eq:nested} and \eqref{vit-1}. In the next lemma we show that for a family of these cylinders also \eqref{vit-2} is satisfied and so the constructed family of intrinsic cylinders can actually provide a suitable Vitali-type covering; \eqref{vit-2} essentially replaces the necessity to work in a metric space, which is most commonly assumed in order to apply Vitali's covering. In some sense, the construction of Lemma~\ref{lem:scal} provides a sort of intrinsic metric with respect to any integrable function.
%%%%%%%%%%%%%
\begin{lemma}\label{lem:scal?}
Let $p\in (\frac{2n}{n+2},2)$, $Q_{2S,2R}\subset E_T$, $f\in L^1(Q_{Q_{2S,2R}})$, and $\hat b\in (0,1)$. For every {$z\equiv(t,x)\in Q_{S,R}$} and $0<s\leq S$, there exist $r(s,z)$, $\lambda_{s,z}$,  and a sub-intrinsic cylinder $Q_{s,r(s,z)}(z)=(t-s,t+s)\times B_{r(s,z)}(x)$, such that all properties of Lemma~\ref{lem:scal} hold. In particular, the cylinders $Q_{s,r(s,z)}(z)$ form  a nested family of sub-intrinsic cylinders with respect to the sizes
\[
\lambda_{s,z}^{2-p}=\frac{s}{(r(s,z))^2}=\theta_{s,z}^\frac{1-m}{m+1},\text{ for }m+1=p.
\] 
Moreover, if we assume that the cylinder $Q_{2S,2R}$ is sub-intrinsic, namely that for some $K\geq 1$
\[
\bigg(\dashiint_{Q_{2S,2R}}\abs{f}\,dxdt\bigg)^\frac{2-p}{p}\leq K\frac{S}{R^2}=:K\lambda_o^{2-p}\, (=:K\theta_o),
\]
then 
\begin{enumerate}
\item \label{scal?:1}
for all $z\equiv(t,x)\in Q_{S,R}$, we have 
\[
\lambda_o^{2-p}\leq \lambda_{S,z}^{2-p}\leq cK^{2p\check{a}}\lambda_o^{2-p}\ \ 
\text{ or }\ \ \theta_o\leq \theta_{S,z}\leq cK^{2p\check{a}}\theta_o
\]
with $\check a=\frac1{2p+n(p-2)}$;
%\]
%\item There is a constant $c_1$, such that  if 
% $Q_{s(r,z),r}(z)\cap Q_{s(r,y),r}(y)\neq\emptyset$, then
%$Q_{s(r,z),r}(z)\subset Q_{s(c_1r,y),c_1r}(y)$ and $Q_{s(r,y),r}(y)\subset Q_{s(c_1r,z),c_1r}(z)$, for each $r\leq c_1R$.
\item \label{scal?:2}
there is a constant $c_1>1$, depending only on $n$, $m$, $\hat b$, such that  if 
$Q_{s,r(s,z)}(z)\cap Q_{s,r(s,y)}(y)\neq\emptyset$, then
\[
Q_{s,r(s,z)}(z)\subset Q_{c_1s,r(c_1s,y)}(y)\ \ \text{ and }\ \  Q_{s,r(s,y)}(y)\subset Q_{c_1s,r(c_1s,z)}(z),
\]
for each $s\leq\frac S{c_o}$, with $c_o=cK^{2p\hat{a}}$.
\end{enumerate}
\end{lemma}
\begin{proof}
For $z=(t,x)\in Q_{S,R}$, we will construct proper sub-intrinsic cylinders.

To prove the estimate on $\lambda_{S,z}^{2-p}=:\theta_{S,z}$, we have to fix the initial cube $Q_{S,R(S,z)}(z)$, required by Lemma~\ref{lem:scal}. We do this, by defining $R(S,z)$ as the maximum in $(0,R]$ such that
\[
\bigg(\int_{t-S/2}^{t+S/2}
  \int_{B_{R(S,z)}(x)}\abs{f}^{p} \,dxdt\bigg)^{2-p}R(S,z)^{2p}\abs{B_{R(S,z)}}^{p-2}\leq S^2.
\]
The above construction implies
\[
\bigg(\dashiint_{Q_{S,R(S,z)}(z)}\abs{f}\,dxdt\bigg)^\frac{2-p}{p}\leq\frac{S}{R(S,z)^2}=:\lambda_{S,z}^{2-p}. 
\]
Since
\begin{align*}
\bigg(\int_{t-S/2}^{t+S/2}
  \int_{B_{R}(x)}\abs{f} \,dxdt\bigg)^{2-p}R^{2p}\abs{B_R}^{p-2}&\leq \bigg(\int_{-2S}^{2S}
  \int_{B_{2R}}\abs{f}\,dxdt\bigg)^{2-p}R^{2p}\abs{B_R}^{p-2}\\ 
&\leq  cK^{p} S^2
\end{align*}
and $K\geq 1$, we find that indeed $R(S,z)\leq R$ and that $\frac{R^{2p}\abs{B_R}^{p-2}}{cK^p}\leq {R(S,z)^{2p}\abs{B_{R(S,z)}}^{p-2}}$, where $c$ depends only on $n$ and $p$. Hence, 
\[
\frac{R}{cK^\frac{p}{2p+n(p-2)}}=c^{-1}K^{-p\check{a}}R \leq R(S,z)\leq R,
\]
which implies the last assertion
\[
\lambda_o^{2-p}\leq \lambda_{S,z}^{2-p}\leq c^2 K^{2p\check{a}}\lambda_o^{2-p}.
\]
This concludes the proof of \eqref{scal?:1}.  

Let us now come to \eqref{scal?:2}. We let $\gamma=c^{-2}K^{-2\check{a}p}$ (notice that $\gamma\in(0,1)$), and
\[
c_o=\max\Bigset{\Big(\frac3{\gamma}\Big)^{1/{\hat b}},\frac{3}{\gamma}};
\]
since $\hat{b}\leq 1$, we have $c_o=\frac{3}{\gamma}$. Take $s\leq \frac{S}{c_o}$; 
without loss of generality, we may assume that $r(s,z)\geq r(s,y)$. 
 Now \eqref{scal:7} of Lemma~\ref{lem:scal} implies that
$$Q_{s,r(s,y)}({y})\subset Q_{3s, 3r(s,z)}(z)\subset Q_{3s,r(3 s,z)}(z),$$
and the second inclusion of \eqref{scal?:2} is satisfied with $c_1=3$.

If $r(s,y)\geq \gamma r(s,z)$, we similarly find 
\[
Q_{s,r(s,z)}({z})\subset Q_{3s, \frac{3r(s,y)}{\gamma}}(y)\subset Q_{\frac{3s}{\gamma},r(\frac{3s}\gamma,y)}(y),
\]
which implies the wanted inclusion with $c_1=\frac3\gamma$. Hence, we are left with the case, where 
\[
r(s,y)< \gamma r(s,z) \leq \Big(\frac{s}{S}\Big)^{\hat{b}}\gamma R\leq \Big(\frac{s}{S}\Big)^{\hat{b}} R(S,y).
\]
In this case, by \eqref{scal:5} and \eqref{scal:6} of Lemma~\ref{lem:scal}, there exists a $\sigma\in[s,S)$ such that $r(s,y)=\Big(\frac{s}{\sigma}\Big)^{\hat b}\tilde{r}(\sigma,y)$ and $Q_{\sigma, r(\sigma,y)}(y)$ is intrinsic. Moreover, since $r(s,z)\leq \Big(\frac{s}{\sigma}\Big)^{\hat b} r(\sigma,z)$ by \eqref{scal:0} of Lemma~\ref{lem:scal}, we have that $r(\sigma,y)=\tilde{r}(\sigma,y)\leq r(\sigma,z)$. 
Now we let $y=(t,x)$ and $z=(t_1,x_1)$ and estimate.

{If $\sigma\in[\frac{S}{3} ,S)$, using that $Q_{\sigma, r(\sigma,y)}(y)$ is intrinsic implies
\begin{align*}
\frac{S^2}{9}\leq \sigma^2 &=\bigg(\int_{t-\sigma/2}^{t+\sigma/2}
  \int_{B_{\tilde{r}(\sigma ,y)}(x)}\abs{f} \,d\xi d\tau\bigg)^{2-p}\abs{B_{\tilde{r}(\sigma,y)}}^{p-2}(\tilde{r}(\sigma,y))^{2p}
	\\
  &\leq \bigg(\iint_{Q_{2R,2S}}\abs{f}dxdt\bigg)^{2-p}\abs{B_{\tilde{r}(\sigma,y)}}^{p-2}(\tilde{r}(\sigma,y))^{2p}
	\\
	  &\leq \bigg(\iint_{Q_{2R,2S}}\abs{f}dxdt\bigg)^{2-p}\abs{B_{r(\sigma,z)}}^{p-2}(r(\sigma,z))^{2p}
	\\
  &\leq \bigg(\int_{Q_{2R,2S}}\abs{f}dxdt\bigg)^{2-p}\abs{B_{R}}^{p-2}R^{2p}
	\\
  &\leq cK^p S^2.
\end{align*}

Since the first and the last member of the inequality are similar, it means that actually $r(s,z)\sim\tilde{r}(\sigma ,y)=r(\sigma ,y)$. More explicitly,
\[
\abs{B_{r(\sigma,z)}}^{p-2}(r(\sigma,z))^{2p}\leq c K^p \abs{B_{\tilde{r}(\sigma,y)}}^{p-2}(\tilde{r}(\sigma,y))^{2p}
\]
and so $r(\sigma,z)\leq {\tilde c_1}r(\sigma,y)$, for ${\tilde c_1}= [cK^p]^{\frac{1}{2p-(2-p)n}}$; moreover,
\[
r(s,z)\leq \Big(\frac{s}{\sigma}\Big)^{\hat b}r(\sigma,z)\leq {\tilde c_1}\Big(\frac{s}{\sigma}\Big)^{\hat b}\tilde{r}(\sigma,y)={\tilde c_1}r(s,y).
\]
On the other hand, if $\sigma\in (s,\frac{S}{3})$, since the cylinders are nested and due to \eqref{scal:7} of Lemma~\ref{lem:scal}, we have that $Q_{\sigma, \tilde{r}(\sigma,y)}\subset Q_{3\sigma, 3\tilde{r}(\sigma,y)}\subset Q_{3\sigma, \tilde{r}(3\sigma,y)}$. Hence, we find
\begin{align*}
\sigma^2&=\bigg(\int_{t-\sigma/2}^{t+\sigma/2}
  \int_{B_{\tilde{r}(\sigma,y)}(x)}\abs{f} \,d\xi d\tau\bigg)^{2-p}\abs{B_{\tilde{r}(\sigma,y)}}^{p-2}(\tilde{r}(\sigma,y))^{2p}
  \\
  &\leq \bigg(\int_{t-\sigma/2}^{t+\sigma/2}\int_{B_{\tilde{r}(\sigma,y)}(x)}\abs{f} \,d\xi d\tau\bigg)^{2-p}\abs{B_{r(\sigma,z)}}^{p-2}(r(\sigma,z))^{2p}
\\
&\leq   
  \bigg(\int_{t_1-\frac{3\sigma}2}^{t_1+\frac{3\sigma}2}
  \int_{B_{3r(\sigma,z)}(x_1)}\abs{f} \,d\xi d\tau\bigg)^{2-p}\abs{B_{r(\sigma,z)}}^{p-2}(r(\sigma,z))^{2p}\\
  &\leq \bigg(\int_{t_1-\frac{3\sigma}2}^{t_1+\frac{3\sigma}2}
  \int_{B_{r(3\sigma,z)}(x_1)}\abs{f} \,d\xi d\tau\bigg)^{2-p}\abs{B_{r(3\sigma,z)}}^{p-2}(r(3\sigma,z))^{2p}\\
	&\leq 9\sigma^2,
\end{align*}
where the last inequality follows by the sub-intrinsic construction of Lemma~\ref{lem:scal}. As before,  
%\textcolor{red}{(there is a slight inconsistency about the height of the cylinders, with respect to the notation)}
\[
r(s,z)\leq \Big(\frac{s}{\sigma}\Big)^{\hat b}r(\sigma,z)\leq {\tilde c_2}\Big(\frac{s}{\sigma}\Big)^{\hat b}\tilde{r}(\sigma,y)={\tilde c_2}r(s,y)
\]
for $\tilde{c}_2=9^\frac{1}{2p-n(2-p)}$.

Therefore, for $s\in (0,{\frac S{c_o}}]$, we find that 
$$Q_{s,r(s,z)}(z)\subset Q_{c_1 s, c_1 r(s,y)}(y) {\,\subset Q_{c_1s, r(c_1r,y)}(y)},$$ 
{with $c_1=\max\set{3,\frac3\gamma,\tilde{c}_1,\tilde{c}_2}$}, which finishes the proof.
 }
\end{proof}

%%%%%%%%%%%%%%%%%%%%%%%%%%%%%%%%%%%%%%
%%%%%%%%%%%%%%%%%%%%%%%%%%%%%%%%%%%%%%
\section{Energy Estimates in General Cylinders}
The main result of this section is Lemma~\ref{lem:osc-energy}, which concerns proper energy estimates.

\begin{lemma}\label{lem:osc-energy}
Let $u \geq 0$ be a local, weak solution to \eqref{PME}-\eqref{PMS-eq:structure} in $E_T$ with $\frac{(n-2)_+}{n+2}< m < 1$. Fix a point $z_{\origin} \in E_{T}$, and, for $\theta > 0$, suppose 
\[
Q_{\theta(2\rho)^2,2\rho}(z_{\origin}) = Q_{\theta(2\rho)^2,2\rho}(x_{\origin}, t_{\origin})\subset E_{T}.
\]
Then there exists a constant $\gamma = \gamma(\data) > 1$ such that, for every $c \geq 0$, we have the \emph{energy estimates}
\begin{equation}\label{PMS-eq: energy}
\begin{aligned}
\esssup_{t \in \Lambda_{\theta\rho^2}(t_{\origin})} &\int_{B_{\rho}(x_{\origin})} \abs{u^{m}-c^{m}}^\frac{m+1}{m}\,dx + \esssup_{t \in \Lambda_{\theta\rho^2}(t_{\origin})}\int_{B_{\rho}(x_{\origin})} \abs{u^{m}-c^{m}}\abs{u-c}\,dx
\\
&+ \iint_{Q_{\theta \rho^2,\rho}(z_{\origin})}  \abs{Du^{m}}^{2}\,dxdt \\
&\quad \leq \frac{\gamma}{\theta\rho^2} \iint_{Q_{\theta(2\rho)^2,2\rho}(z_{\origin})} \abs{u^{m}-c^{m}} \abs{u-c}\,dxdt \\
&\quad \quad + \frac{\gamma}{\rho^{2}} \iint_{Q_{\theta(2\rho)^2,2\rho}(z_{\origin})}\abs{u^{m}-c^{m}}^2\,dxdt\\ 
&\quad\quad + \gamma \iint_{Q_{\theta(2\rho)^2,2\rho}(z_{\origin})}\rho^2|f|^2\,dxdt,
\end{aligned}
\end{equation}
and
\begin{equation}\label{PMS-eq: energy modified}
\begin{aligned}
\esssup_{t \in \Lambda_{\theta\rho^2}(t_{\origin})} & \int_{B_{\rho}(x_{\origin})} \abs{u^{m}-c^{m}}^\frac{m+1}{m}\,dx  + \iint_{Q_{\theta\rho^2,\rho}(z_{\origin})} \abs{Du^{m}}^{2}\,dxdt\\
&\leq \frac{\gamma}{\theta\rho^2} \iint_{Q_{\theta(2\rho)^2,2\rho}(z_{\origin})}  \abs{u-c}^{m + 1}\,dxdt + \frac{\gamma}{\rho^{2}} \iint_{Q_{\theta(2\rho)^2,2\rho}(z_{\origin})}\abs{u-c}^{2m}\,dxdt\\ 
&\quad\quad + \gamma \iint_{Q_{\theta(2\rho)^2,2\rho}(z_{\origin})}\rho^2|f|^2\,dxdt.
\end{aligned}
\end{equation}
\end{lemma}
\begin{remark}
{\normalfont Since we are taking $f\in L^\infty_{\loc}(E_T)$, the last term on the right-hand side both of \eqref{PMS-eq: energy} and \eqref{PMS-eq: energy modified} could be further estimated with the $L^\infty$-norm;  at this stage this general statement suffices.}
\end{remark}
\begin{proof}
Estimate \eqref{PMS-eq: energy modified} can be easily deduced from \eqref{PMS-eq: energy}, taking into account that for $a, b \geq 0$, and for $0< m < 1$ we have
\begin{equation}\label{PMS-eq: useful inequality}
\abs{a^{m} - b^{m}} \leq \abs{a - b}^{m}.
\end{equation}

After a change of coordinates we can assume $ z_{\origin} = (0,0)$.
Choose a point $t_{1} \in \Lambda_{\theta(2\rho)^2}$ and in the weak formulation \eqref{PMS-eq: weak sol mollified} take the test function
\begin{equation*}
\phi \deq  \pm \pa{u_{\eps}^{m} -c^{m}}_{\pm}\,\eta^{2} \psi_{\eps}^{h}, %
\end{equation*}
where:
\begin{itemize}
\item $\eta(x,t) \in C^{\infty}_{0}(\R^{N+1})$ is a cut-off function such that
\begin{gather*}
\supp \eta \subset Q_{\theta(2\rho)^2,2\rho}, \qquad 0 \leq \eta \leq 1, \qquad \eta= 1\  \text{ on $Q_{\theta\rho^2,\rho}$},\\
\rho \abs{D\eta} + \theta \rho^2 \abs{\eta_{t}} \leq 1000.
\end{gather*}
\item $\psi^{h}(t)$ is a piecewise linear approximation of the characteristic function, such that
\[
\psi^{h}(t) = 
\begin{cases}
\ 1 &\quad  t \in [- \theta (2\rho)^2 + h, t_{1} -h]\\
\ 0 &\quad  t \in \left( - \theta (2\rho)^2, - \theta (2\rho)^2 + \frac{h}{10}\right] \cup \left[t_{1} - \frac{h}{10}, t_{1}\right)
\end{cases}
\]
and
\[
\left\lvert \frac{d\psi^{h}}{dt}\right\rvert \leq \frac{10}{9h},
\]
\item $u_{\eps}$ and $\psi^{h}_{\eps}$ are the mollification of $u$ and $\psi^{h}$ in time direction for $\eps \leq \frac{h}{20}$. 
\end{itemize}

\medskip

\noindent\textbf{Estimates on the set $[u \geq c]$}: we integrate by parts the first term on the left-hand side of \eqref{PMS-eq: weak sol mollified}, and we have
\begin{align}\label{PMS-eq: integration parts +}
-&\int_{-\theta(2\rho)^2}^{t_{1}}  \int_{B_{2\rho}\cap [u_{\eps} \geq c]} u_{\eps}\phi_{t}\,dxdt \notag\\
&= \int_{-\theta(2\rho)^2}^{t_{1}} \int_{B_{2\rho}\cap [u_{\eps} \geq c]}  \tder{u_{\eps}}  \pa{u_{\eps}^{m} - c^{m}}\, \eta^{2}\psi_{\eps}^{h}\,dxdt \notag\\
&= \int_{-\theta(2\rho)^2}^{t_{1}} \int_{B_{2\rho}\cap [u_{\eps} \geq c]} \tder{}\pa{\int_{a}^{u_{\eps}} \pa{y^{m} -c^{m}}\,dy}  \eta^{2}\psi_{\eps}^{h}\,dxdt.
\end{align}
Integrating by parts \eqref{PMS-eq: integration parts +} and taking the limits, first  as $\eps \to 0$ and then as $h \to 0$, we end up with
\begin{align}\label{PMS-eq: limits +}
-&\int_{-\theta(2\rho)^2}^{t_{1}}  \int_{B_{2\rho}\cap [u_{\eps} \geq c]}  u_{\eps}\phi_{t}\,dxdt\notag\\
&\longrightarrow \int_{B_{2\rho} \cap [u \geq c]}\pa{\int_{c}^{u} \pa{y^{m} -c^{m}}\,dy} \eta^{2}(x,t_{1})\,dx \notag\\
&\quad -2\int_{-\theta(2\rho)^2}^{t_{1}} \int_{B_{2\rho}\cap [u \geq c]} \pa{\int_{c}^{u} \pa{y^{m} -c^{m}}\,dy}\eta \eta_{t}\,dxdt,
\end{align}
since by definition $\eta = 0$ at $t = -\theta(2\rho)^2$. Notice that the term with the time derivative of $\psi^{h}_{\eps}$ does not appear: this follows from the bound on $\frac{d\psi^{h}}{dt}$ and the fact that we are taking the limits \emph{first} as $\eps \to 0$ and \emph{then} as $h \to 0$. 

We estimate the first term on the right-hand side of \eqref{PMS-eq: limits +}. We use Jensen's inequality, \eqref{PMS-eq: bound integral +}  and \eqref{PMS-eq: useful inequality} to obtain
\begin{align*}
&\int_{B_{2\rho} \cap [u \geq c]} \pa{ \int_{c}^{u} \pa{y^{m} -c^{m}}\,dy }\eta^{2}(x,t_{1})\,dx\geq {\frac{1}{2}} \int_{B_{2\rho}\cap [u \geq c]} \pa{u -c} \pa{u^{m} -c^{m}}\eta^{2}(x,t_{1})\,dx\\
&\ge{\frac{1}{4}} \int_{B_{2\rho}\cap [u \geq c]} \pa{u -c} \pa{u^{m} -c^{m}}\eta^{2}(x,t_{1})\,dx+
 {\frac{m}{4}} \int_{B_{2\rho}\cap [u \geq c]}  \pa{u^{m} -c^{m}}^{\frac{m + 1}{m}}\eta^{2}(x,t_{1})\,dx\\
&={\frac{1}{4}} \int_{B_{2\rho}\cap [u \geq c]} \abs{u -c} \abs{u^{m} -c^{m}}\eta^{2}(x,t_{1})\,dx
+{\frac{m}{4}} \int_{B_{2\rho} \cap [u \geq c]} \abs{u^{m} -c^{m}}^{\frac{m + 1}{m}}\eta^{2}(x,t_{1})\,dx.
\end{align*}

The second term on the right-hand side of \eqref{PMS-eq: limits +} is estimated by
\begin{multline*}
\int_{-\theta(2\rho)^2}^{t_{1}} \int_{B_{2\rho}\cap [u \geq c]} \pa{\int_{c}^{u}  \pa{y^{m} -c^{m}}\,dy }\,\eta\eta_{t}\,dxdt\\
\leq \int_{-\theta(2\rho)^2}^{t_{1}} \int_{B_{2\rho}\cap [u \geq c]}\abs{u-c} \abs{u^{m}-c^{m}}\eta \abs{\eta_{t}}\,dxdt.
\end{multline*}

Next we consider the term with the gradient in \eqref{PMS-eq: weak sol mollified}. We differentiate $\phi$ and take the limits, first as $\eps \to 0$ and then as $h \to 0$; collecting all the estimates yields
\begin{align*}
& \int_{B_{2\rho}\cap [u \geq c]} \abs{u -c} \abs{u^{m} -c^{m}}\eta^{2}(x,t_{1})\,dx+\int_{B_{2\rho} \cap [u \geq c]}  \abs{u^{m} -c^{m}}^{\frac{m+1}{m}}\eta^{2}(x,t_{1})\,dx\\
&\quad+ \int_{-\theta(2\rho)^2}^{t_{1}} \int_{B_{2\rho}\cap [u \geq c]} \eta^{2} \A(x,t,u,Du^m) \cdot D{\pa{u^{m} -c^{m}}}\,dxdt\\
&\leq \gamma\int_{-\theta(2\rho)^2}^{t_{1}} \int_{B_{2\rho}\cap [u \geq c]}\abs{u-c} \abs{u^{m}-c^{m}}\eta \abs{\eta_{t}}\,dxdt\\
&\quad + \gamma  \int_{-\theta(2\rho)^2}^{t_{1}} \int_{B_{2\rho}\cap [u \geq c]}  \abs{u^{m} -c^{m}} \eta \abs{ \A(x,t,u,Du^m)} \abs{ D\eta}\,dxdt\\
&\quad +  \int_{-\theta(2\rho)^2}^{t_{1}} \int_{B_{2\rho}\cap [u \geq c]} \abs{u^{m} -c^{m}} \eta^2 |f|\,dxdt.
\end{align*}

Using the lower bound on $\A$ we obtain
\begin{align*}
&\int_{-\theta(2\rho)^2}^{t_{1}}  \int_{B_{2\rho}\cap [u \geq c]}  \eta^{2}\A(x,t,u,Du^m) \cdot D{\pa{u^{m} 
-c^{m}}}\,dxdt\\
&= \int_{-\theta(2\rho)^2}^{t_{1}} \int_{B_{2\rho}\cap [u \geq c]}  \eta^{2} \A(x,t,u,Du^m) \cdot Du^{m}\,dxdt\\
& \geq \nu \int_{-\theta(2\rho)^2}^{t_{1}} \int_{B_{2\rho}\cap [u \geq c]} \eta^{2}\abs{Du^{m}}^{2}\,dxdt.
\end{align*}
The upper bound and Young's inequality imply
\begin{align*}
&\int_{-\theta(2\rho)^2}^{t_{1}} \int_{B_{2\rho}\cap [u \geq c]} \abs{u^{m} -c^{m}} \eta  \abs{\A(x,t,u,Du^m)} \abs{D\eta}\,dxdt\\
&\leq L \int_{-\theta(2\rho)^2}^{t_{1}} \int_{B_{2\rho}\cap [u \geq c]} \abs{u^{m} -c^{m}} \eta \abs{Du^{m}} \abs{D\eta}\,dxdt\\
&\leq \frac{\nu}{2} \int_{-\theta(2\rho)^2}^{t_{1}} \int_{B_{2\rho}\cap [u \geq c]} \eta^{2} \abs{Du^{m}}^{2} \,dxdt\\
&\quad + \gamma(\nu, L)\int_{-\theta(2\rho)^2}^{t_{1}} \int_{B_{2\rho}\cap [u \geq c]} \abs{u^{m} -c^{m}}^2\abs{D\eta}^{2} \,dxdt.
\end{align*}
%%%%%%%%%%%%%%%%%
In this way we end up with
\begin{equation}\label{PMS-eq: contr +}
\begin{aligned}
& \int_{B_{2\rho}\cap [u \geq c]} \abs{u -c} \abs{u^{m} -c^{m}}\eta^{2}(x,t_{1})\,dx\\
&+\int_{B_{2\rho} \cap [u \geq c]} \abs{u^{m} -c^{m}}^{\frac{m+1}{m}}\eta^{2}(x,t_{1})\,dx + \int_{-\theta(2\rho)^2}^{t_{1}} \int_{B_{2\rho}\cap [u \geq c]} \eta^{2} \abs{Du^{m}}^{2}\,dxdt\\
&\leq \gamma\int_{-\theta(2\rho)^2}^{t_{1}} \int_{B_{2\rho}\cap [u \geq c]}\abs{u-c} \abs{u^{m}-c^{m}}\eta \abs{\eta_{t}}\,dxdt\\
&\quad + \gamma \int_{-\theta(2\rho)^2}^{t_{1}} \int_{B_{2\rho}\cap [u \geq c]} \abs{u^{m} -c^{m}}^2\abs{D\eta}^{2} \,dxdt\\
&\quad +  \int_{-\theta(2\rho)^2}^{t_{1}} \int_{B_{2\rho}\cap [u \geq c]} \abs{u^{m} -c^{m}} \eta^2 |f|\,dxdt.
\end{aligned}
\end{equation}
%%%%%%%%%%%%%%%%%
If $c = 0$, the proof is finished by a straightforward application of H\"older's inequality. If $c>0$, we have to consider also the  set $[u<c]$.

\medskip

\noindent\textbf{Estimates on the set $[u < c]$}: integrating by parts the first term on the left-hand side of \eqref{PMS-eq: weak sol mollified} we have
\begin{align}\label{PMS-eq: integration parts -}
-& \int_{-\theta(2\rho)^2}^{t_{1}}  \int_{B_{2\rho}\cap [u_{\eps} < c]} u_{\eps}\phi_{t}\,dxdt \notag\\
&= \int_{-\theta(2\rho)^2}^{t_{1}} \int_{B_{2\rho}\cap [u_{\eps} < c]}  -\tder{u_{\eps}} \pa{c^{m} - u_{\eps}^{m}}\, \eta^{2}\psi_{\eps}^{h}\,dxdt  \\
&= \int_{-\theta(2\rho)^2}^{t_{1}} \int_{B_{2\rho}\cap [u_{\eps} < c]} \tder{}\pa{\int_{u_{\eps}}^{c} \pa{c^{m} - y^{m}}\,dy}  \eta^{2}\psi_{\eps}^{h}\,dxdt.\notag
\end{align}

After another integration by parts and taking the limits first as $\eps \to 0$, and then as $h \to 0$, we conclude
\begin{align}\label{PMS-eq: limits -}
-&\int_{-\theta(2\rho)^2}^{t_{1}} \int_{B_{2\rho}\cap [u_{\eps} < c]}  u_{\eps}\phi_{t}\,dxdt\notag\\
&\longrightarrow \int_{B_{2\rho} \cap [u < c]}\pa{\int_{u}^{c} \pa{c^{m} - y^{m}}\,dy} \eta^{2}(x,t_{1})\,dx \\
&\quad -2\int_{-\theta(2\rho)^2}^{t_{1}} \int_{B_{2\rho}\cap [u < c]} \pa{\int_{u}^{c} \pa{c^{m} - y^{m}}\,dy}\eta \eta_{t}\,dxdt,\notag
\end{align}
since, as before, $\eta$ vanishes at $t = -\theta(2\rho)^2$ and the term with $\frac{d\psi^{h}_{\eps}}{dt}$ goes to $0$.

We control the first term using \eqref{PMS-eq: bound integral -} and \eqref{PMS-eq: useful inequality}:
\begin{align*}
&\int_{B_{2\rho} \cap [u < c]} \pa{\int_{u}^{c} \pa{c^{m} - y^{m}}\,dy}\eta^{2}(x,t_{1})\,dx\geq {\frac{m}{2}} \int_{B_{2\rho} \cap [u < c]} \pa{c - u} \pa{c^{m} - u^{m}}\,\eta^{2}(x,t_{1})\,dx\\
&\ge {\frac{m}{4}} \int_{B_{2\rho} \cap [u < c]} \pa{c - u} \pa{c^{m} - u^{m}}\,\eta^{2}(x,t_{1})\,dx+{\frac{m}{4}} \int_{B_{2\rho} \cap [u < c]}  \pa{c^{m} -u^{m}}^{\frac{m + 1}{m}}\eta^{2}(x,t_{1})\,dx\\
&= {\frac{m}{4}} \int_{B_{2\rho} \cap [u < c]} \abs{u - c} \abs{u^{m} - c^{m}}\,\eta^{2}(x,t_{1})\,dx+{\frac{m}{4}} \int_{B_{2\rho} \cap [u < c]}  \abs{u^{m} -c^{m}}^{\frac{m + 1}{m}}\eta^{2}(x,t_{1})\,dx.
\end{align*}
 
The other calculations with the spatial gradient are similar to those on the set $[u \geq c]$, and finally give
\begin{equation}\label{PMS-eq: contr -}
\begin{aligned}
&\int_{B_{2\rho} \cap [u < c]} \abs{u - c} \abs{u^{m} - c^{m}}\,\eta^{2}(x,t_{1})\,dx+\int_{B_{2\rho} \cap [u < c]} \abs{u^{m} - c^{m}}^{\frac{m+1}{m}}\eta^{2}(x,t_{1})\,dx\\
&\quad + \int_{-\theta(2\rho)^2}^{t_{1}} \int_{B_{2\rho}\cap [u < c]} \abs{Du^{m}}^{2}\,\eta^{2}\,dxdt\\
&\leq \gamma\int_{-\theta(2\rho)^2}^{t_{1}} \int_{B_{2\rho}\cap [u < c]}\abs{u-c} \abs{u^{m}-c^{m}}\,\eta \abs{\eta_{t}}\,dxdt \\
&\quad + \gamma \int_{-\theta(2\rho)^2}^{t_{1}} \int_{B_{2\rho}\cap [u < c]} \abs{u^{m} -c^{m}}^2
\abs{D\eta}^{2} \,dxdt\\
&\quad +  \int_{-\theta(2\rho)^2}^{t_{1}} \int_{B_{2\rho}\cap [u < c]} \abs{u^{m} -c^{m}} \eta^2 |f|\,dxdt.
\end{aligned}
\end{equation}

Now if we sum the contributions \eqref{PMS-eq: contr +} and \eqref{PMS-eq: contr -}, notice that
\[
\begin{aligned}
&\int_{-\theta(2\rho)^2}^{t_{1}} \int_{B_{2\rho}\cap [u \ge c]} \abs{u^{m} -c^{m}} \eta^2 |f|\,dxdt\\
&\le\frac1{2\rho^2}\int_{-\theta(2\rho)^2}^{t_{1}} \int_{B_{2\rho}\cap [u \ge c]} \abs{u^{m} -c^{m}}^2 +\frac{\rho^2}2\int_{-\theta(2\rho)^2}^{t_{1}} \int_{B_{2\rho}\cap [u \ge c]} |f|^2\,dxdt,\\
&\int_{-\theta(2\rho)^2}^{t_{1}} \int_{B_{2\rho}\cap [u < c]} \abs{u^{m} -c^{m}} \eta^2 |f|\,dxdt\\
&\le\frac1{2\rho^2}\int_{-\theta(2\rho)^2}^{t_{1}} \int_{B_{2\rho}\cap [u < c]} \abs{u^{m} -c^{m}}^2 +\frac{\rho^2}2\int_{-\theta(2\rho)^2}^{t_{1}} \int_{B_{2\rho}\cap [u < c]} |f|^2\,dxdt,
\end{aligned}
\]
and use the property of $\eta$ and the arbitrariness of $t_{1}$, we conclude the proof.
\end{proof}
A direct consequence of \eqref{PMS-eq: energy modified} for $c=0$ is the following corollary, that is satisfied for sub-intrinsic cylinders:
\begin{corollary}[sub-intrinsic]\label{cor:deg-subint}
Let $u \geq 0$ be a local, weak solution to \eqref{PME}-\eqref{PMS-eq:structure} in $E_T$ with $\frac{(n-2)_+}{n+2}< m < 1$. Fix a point $z_{\origin} \in E_{T}$, and, for $s,\theta > 0$, suppose 
\[
Q_{s,\sqrt{s/\theta}}(z_{\origin}) = Q_{s,\sqrt{s/\theta}}(x_{\origin}, t_{\origin})\subset E_{T},
\]
and
\[
\bigg(\dashiint_{Q_{s,\sqrt{s/\theta}}} u^{m+1}\, dxdt\bigg)^\frac{1-m}{m+1}\leq K\theta,
\]
for some $K\ge1$. Then there exists a constant $\gamma_K = \gamma(\data,K) > 1$ such that 
\begin{align*}
\frac{1}{s}\esssup_{t \in \Lambda_{s/2}(t_{\origin})} &\dashint_{B_{\sqrt{s/4\theta}}(x_{\origin})} u^{m+1}\,dx
+ \dashiint_{Q_{s/2,\sqrt{s/4\theta}}(z_{\origin})}  \abs{Du^{m}}^{2}\,dxdt \\
&\quad \leq \gamma_K\left[\frac{\theta^\frac{m+1}{1-m}}{s}+{\frac s\theta}\,\dashiint_{Q_{s,\sqrt{s/\theta}}(z_{\origin})}|f|^2\,dxdt\right]\\
&\quad \leq \gamma_K\left[\frac{\theta^\frac{m+1}{1-m}}{s}+\sup_{Q_{s,\sqrt{s/\theta}}(z_{\origin})} 
\left[{\frac s\theta}\,|f|^2\right]\right].
\end{align*}
\end{corollary}
We will also need the following energy estimates for proper truncations of $u$. They can be seen as a corollary of Lemma~\ref{lem:osc-energy}. They have been proved in \cite[Section~B.3]{DiBGiaVes11}, to which we refer for all the details.

\begin{proposition}%%\label{Prop:B:3:1}
Let $u\ge0$ be a local, weak subsolution to
\eqref{PME}--\eqref{PMS-eq:structure} in $E_T$ with $\frac{(n-2)_+}{n+2}< m < 1$, and consider 
the truncations 
\begin{equation*}%%\label{Eq:B:3:1bis}
(u^m-k^m)_+\quad\text{ for }\quad k>0.
\end{equation*}
There exists a positive constant $\gm=\gm(m,n,\nu,L)$, 
such that for every cylinder
\[
Q_{\theta(2\rho)^2,2\rho}(z_{\origin}) = Q_{\theta(2\rho)^2,2\rho}(x_{\origin}, t_{\origin})\subset E_{T},
\]
every $k>0$, and every nonnegative, 
piecewise smooth cutoff function $\zeta$ vanishing on 
$\pl B_{2\rho} (x_o)$, 
\begin{equation}\label{Eq:B:3:1}
\begin{aligned}
\essup_{t_o-\theta(2\rho)^2<t\le t_o}&\frac1{m+1}
\int_{B_{2\rho}(x_o)}\umkmp^{\frac{m+1}m}\zeta^2(x,t)dx\\
&\quad -\int_{B_{2\rho}(x_o)}\int_k^u(s^m-k^m)ds 
\zeta^2(x,t_o-\theta(2\rho)^2)dx\\
&\quad + \frac{\nu}4\iint_{Q_{\theta(2\rho)^2,2\rho}(z_{\origin}))}|D\umkmp|^2
 \zeta^2\,dxdt\\
&\le \gm\iint_{Q_{\theta(2\rho)^2,2\rho}(z_{\origin})}u^{m+1}\chi_{[\ukp>0]}
\zeta|\zeta_t|\,dxdt\\
&\quad+\gm\iint_{Q_{\theta(2\rho)^2,2\rho}(z_{\origin})}
\umkmp^2|D\zeta|^2\,dxdt\\
&\quad+\gm\iint_{Q_{\theta(2\rho)^2,2\rho}(z_{\origin})}
\zeta^2\umkmp|f|\,dxdt.
\end{aligned}
\end{equation}
%Analogous estimates hold in ``forward'' cylinders 
%$(y,s)+Q_\rho^{+}(\theta)\subset E_T$. 
\end{proposition}
%%%%%%%%%%%%%%%%%%%%%%%%%%%%%%%
\section{Estimate for the Solution}
In the following, we refer to the cylinders built in Lemma~\ref{lem:scal}. In particular, if $s$ denotes the height of the cylinder and $\theta(s)$ is the corresponding scaling parameter, then $r(s)$ is given by
\begin{equation}\label{eq:radius-s}
r(s):=\sqrt{\frac s{\theta(s)}}.
\end{equation}
For simplicity, whenever possible, instead of $\theta(s)$, $r(s)$, we will write $\theta_s$, $r_s$. Moreover, $Q_s^{\theta_s}$ denotes the cylinder of height $s$, scaling parameter $\theta_s$, and related radius $r_s$ given by \eqref{eq:radius-s}; when there is no possibility of misunderstanding, we will simply write $Q_s$.
\begin{lemma}\label{lm:intr-cons}
Assume that the cylinder $Q_s^{\theta_s}$ is intrinsic. Then for any $a\in(1,2]$ also the cylinder $Q_{as}^{\theta_{as}}$ is intrinsic.
\end{lemma}
\begin{proof}
Assume that $Q_s^{\theta_s}$ is intrinsic: this means that, up to a constant $K\ge1$,
\[
\theta_s=\left[\dashiint_{Q_s^{\theta_s}} u^{1+m}\,dxd\tau\right]^{\frac{1-m}{1+m}}.
\]
Now, for any $a\in(1,2]$
\begin{align*}
\theta_s&\le c_1\left[\dashiint_{Q_{as}^{\theta_{as}}} u^{1+m}\,dxd\tau\right]^{\frac{1-m}{1+m}}\\
&\le c_2\theta_{as}\quad\text{ by \eqref{scal:2} of Lemma~\ref{lem:scal}}\\
&\le c_3\theta_s\quad\text{ by \eqref{scal:3} of Lemma~\ref{lem:scal}}\\
&\le c_4\left[\dashiint_{Q_{as}^{\theta_{as}}} u^{1+m}\,dxd\tau\right]^{\frac{1-m}{1+m}}.
\end{align*}
\end{proof}
Boundedness estimate are well-known in the context of the singular porous medium equation (see, for example, 
\cite[Appendix~B]{DiBGiaVes11}). Here we show that they have a simple expression, when one considers \emph{intrinsic} cylinders. As a consequence, we are then able to deduce a very useful reverse H\"older inequality for powers of the solution $u$, again in intrinsic cylinders. 
\begin{proposition}\label{Prop:B:4:1} 
Let $u\ge0$ be a locally bounded, local, 
weak solution to \eqref{PME}--\eqref{PMS-eq:structure} 
 in $E_T$ with $\frac{(n-2)_+}{n+2}< m < 1$.
Fix $t>0$, assume that $f\in L^\infty_{\loc}(E_T)$, the cylinder $B_{r(t)}(x_o)\times(t_o,t_o+t]\subset E_T$ is 
one of the intrinsic cylinders built in Lemma~\ref{lem:scal} with respect to $u$, and  
\[
\sup_{Q_{2t}}\left[r^2(2t)|f|^2\right]\le\frac c{t}\left[\theta(2t)\right]^{\frac{1+m}{1-m}}
\]
for some $c>0$, where $Q_{2t}$ and $\theta(2t)$ are defined below.
Then there exists a positive constant $\gm$ depending only on 
the data $\datam$, such that for  $Q_{2t}:=B_{r(2t)}(x_o)\times(t_o-t,t_o+t]\subset E_T$ and for any $p>0$, 
we have 
\begin{align*}
\sup_{B_{r(t)}(x_o)\times[t_o,t_o+t]}u &\le\gm
\frac1{(1-\sig)^q}\left(\dashiint_{Q_{2t}} 
u^p\,dxd\tau\right)^{\frac{1}{p}},
\end{align*}
where $q$ depends only on  $p$, $n$, $m$.
\end{proposition}
\begin{remark}
{\normalfont The intrinsic cylinders of Lemma~\ref{lem:scal} are centered at $(t_o,x_o)$, whereas here we consider cylinders which have \emph{a lower vertex} at $(t_o,x_o)$. However, Lemma~\ref{lem:scal} can be adapted to cover this case, or, alternatively, we can deal with centered cylinders, and correspondingly rephrase Proposition~\ref{Prop:B:4:1}.}
\end{remark}
\begin{proof} 
We follow the same strategy used for \cite[Proposition~B.4.1]{DiBGiaVes11}. Without loss of generality, we can assume $(x_o,t_o)\equiv(0,0)$. Fix $t>0$, consider the cylinder
\[
Q_t^{\theta_t}=B_{r(t)}\times(0,t]
\]
and assume it to be intrinsic. By Lemma~\ref{lm:intr-cons}, for any $\sigma\in(0,1]$ we have that all the cylinders
\[
Q_{(1+\sigma)t}^{\theta_{(1+\sigma)t}}=B_{r((1+\sigma)t)}\times(-\sigma t,t]
\]
are intrinsic. From line to line the intrinsic constants may be different, but they are stable. For any $\sigma\in(0,1]$ by \eqref{eq:radius-s} we have
\[
\theta((1+\sigma)t)\,r^2((1+\sigma)t)=(1+\sigma)t.
\]
For fixed 
$\sig\in(0,1]$ and $j=0,1,2,\dots$ set
\begin{equation*}
\begin{array}{ll}
{\dsty t_j=-\sig t-\frac{1-\sig}{2^j}t,\qquad }
&{\dsty \rho_j=r(-t_j+t)}=r((1+\sig)t+\frac{1-\sig}{2^j}t),\\
{}\\
{\dsty B_j=B_{\rho_j},\qquad }&{\dsty Q_j=B_j\times(t_j,t).}
\end{array}
\end{equation*}
This is a family of nested and shrinking cylinders with common 
``vertex'' at $(0,t)$, and by construction 
\begin{equation*}
Q_o=B_{r(2t)}\times(-t,t]\quad\hbox{\rm and}
\quad Q_\infty=B_{r((1+\sig)t)}\times(-\sig t,t]. 
\end{equation*}
Having assumed that $u$ is locally bounded in $E_T$, set 
\begin{equation*}
M=\essup_{Q_o} u,\quad 
M_\sig=\essup_{Q_\infty} u.
\end{equation*}
We first find a relationship between $M$ and $M_\sig$. 
Denote by $\zeta$ a nonnegative, piecewise smooth cutoff 
function in $Q_j$ that equals one on $Q_{j+1}$, and has the form 
$\zeta(x,t)=\zeta_1(x)\zeta_2(t)$, where
\begin{equation*}
\begin{array}{lc}
{\dsty 
\zeta_1=\left\{
\begin{array}{ll}
1\>&\text{ in }\> B_{j+1}\\
0\>&\text{ in }\>{\mathbb R}^n-B_j
\end{array}\right .}\>
&{\dsty |D\zeta_1|\le\frac{1}{\rho_j-\rho_{j+1}}},\\
{}\\
{\dsty 
\zeta_2=\left\{
\begin{array}{ll}
0\>&\text{ for }\> t\le t_j\\
1\>&\text{ for }\> t\ge t_{j+1}
\end{array}\right .}\>
&{\dsty 0\le \zeta_{2,t}\le\frac{2^{j+1}}{(1-\sig)t};}
\end{array}
\end{equation*} 
introduce the increasing sequence of levels 
\begin{equation*}
k_j=k-{\txty \frac1{2^{j}}}k
\end{equation*} 
where $k>0$ is to be chosen. Estimates (\ref{Eq:B:3:1}) 
with $\umkmnup$ yield
\begin{equation}\label{Eq:B:4:2}
\begin{aligned}
\sup_{t_j\le \tau\le t}&\int_{B_j}[\umkmnup
\zeta]^{\frac{m+1}m}(x,\tau)\,dx\\
&\qquad+\frac{\nu}4\iint_{Q_j}|D [\umkmnup\zeta]|^2\,dxd\tau\\
&\le\gm\iint_{Q_j} u^{m+1}\chi_{[\umkmnup>0]}\zeta\zeta_\tau\,dxd\tau\\
&\qquad+\gm\iint_{Q_j}\umkmnup^2|D\zeta|^2\,dxd\tau\\
&\qquad+\gm\iint_{Q_j}\zeta^2\umkmnup |f| \,dxd\tau.
\end{aligned}
\end{equation}
%If $\frac{(n-2)_+}{n+2}<m<1$, it means we are taking $\lm_r>0$ with $r\in[1,\frac{2n}{n+2}]$.
In the estimations below repeated use is made 
of the inequality 
\begin{equation*}
|[u>k_{j+1}]\cap Q_j| \le\gm\frac{2^{(j+1)s}}{k^{ms}}
\iint_{Q_j}(u^m-k_j^m)_+^s\,dxd\tau
\end{equation*}
valid for all $s>0$. Then estimate
{\small
\begin{align*}
%\iint_{Q_j}\umkmnup^{\frac{m+1}m}\zeta\zeta_\tau dx\,d\tau
%&\le\gm\frac{2^{2j}}{(1-\sig)t}
%\iint_{Q_j}\umkmnp^{\frac{m+1}m}dx\,d\tau\\
\iint_{Q_j}u^{m+1}\chi_{[\umkmnup>0]}\zeta\zeta_\tau\,dxd\tau
&\le\gm\frac{2^{(2+\frac{m+1}m)j}}{(1-\sig)t}
\iint_{Q_j}\umkmnp^{\frac{m+1}m}\,dxd\tau\\
\iint_{Q_j}\umkmnup^2|D\zeta|^2\,dxd\tau
&\le\gm\frac{2^{\frac{m+1}m j}}{(\rho_j-\rho_{j+1})^2}
\frac1{k^{1-m}}\iint_{Q_j}\umkmnp^{\frac{m+1}m}\,dxd\tau\\
&\le\gm\frac{2^{\frac{m+1}m j}}{(1-\sig)t}2^{2j}
\frac{\theta(2t)}{k^{1-m}}\iint_{Q_j}\umkmnp^{\frac{m+1}m}\,dxd\tau.
\end{align*}
}
Moreover,
\begin{align*}
\iint_{Q_j}\zeta^2\umkmnup |f| \,dxd\tau\le&\frac12\iint_{Q_j}\frac{\umkmnup^2}{(\rho_j-\rho_{j+1})^2}\,dxd\tau\\
&+\frac12\iint_{Q_j}\rho_j^2|f|^2\chi_{[\umkmnup>0]}\,dxd\tau=I + II.
\end{align*}
The first term on the right-hand side $I$ can be estimated as before, namely
\[
I\le \gm\frac{2^{\frac{m+1}m j}}{(1-\sig)t}2^{2j}
\frac{\theta(2t)}{k^{1-m}}\iint_{Q_j}\umkmnp^{\frac{m+1}m}\,dxd\tau.
\]
As for the second term $II$, we have
\begin{align*}
\iint_{Q_j}\rho_j^2|f|^2\chi_{[\umkmnup>0]}\,dxd\tau&\le\iint_{Q_j}r^2(2t)|f|^2\chi_{[\umkmnup>0]}\,dxd\tau\\
&\le\sup_{Q_{2t}}\left[r^2(2t)|f|^2\right]\,|[u>k_{j+1}]\cap Q_j|\\
&\le\gm\sup_{Q_{2t}}\left[r^2(2t)|f|^2\right]\,\frac{2^{(j+1)\frac{m+1}{m}}}{k^{m+1}}
\iint_{Q_j}(u^m-k_j^m)_+^{\frac{m+1}{m}}\,dxd\tau\\
&\le\gm\frac{2^{\frac{m+1}{m}j}}{(1-\sig)t}\,\left[\frac{\theta(2t)}{k^{1-m}}\right]^{\frac{1+m}{1-m}}
\iint_{Q_j}(u^m-k_j^m)_+^{\frac{m+1}{m}}\,dxd\tau.
\end{align*}
Combining these estimates,
(\ref{Eq:B:4:2}) yields
\begin{align*}
\sup_{t_j\le \tau\le t}&\int_{B_j}[\umkmnup\zeta]^{\frac{m+1}m}(x,\tau)\,dx +
\iint_{Q_j}|D [\umkmnup\zeta]|^2\,dxd\tau\\
&\le\gm \frac{2^{\frac{3(m+1)}m j}}{(1-\sig) t}
\Bigg[1+ \frac{\theta(2t)}{k^{1-m}}+\left[\frac{\theta(2t)}{k^{1-m}}\right]^{\frac{1+m}{1-m}}\Bigg]
\iint_{Q_j}\umkmnp^{\frac{m+1}m}\,dxd\tau.
\end{align*}
The last term in $[\cdots]$ is estimated by stipulating to take
\begin{equation}\label{Eq:B:4:3} 
\frac{\theta(2t)}{k^{1-m}}\le1\ \ \Rightarrow\ \ k\ge\left(\theta(2t)\right)^{\frac{1}{1-m}}.
\end{equation}
With this stipulation, the previous inequality implies
\begin{equation}\label{Eq:B:4:4}
\begin{aligned}
\sup_{t_j\le \tau\le t}&\int_{B_j}[\umkmnup\zeta]^{\frac{m+1}m}(x,\tau)\,dx +
\iint_{Q_j}|D [\umkmnup\zeta]|^2\,dxd\tau\\
&\le\frac{\gm 2^{\frac{3(m+1)}m j}}{(1-\sig) t}
\iint_{Q_j}\umkmnp^{\frac{m+1}m}\,dxd\tau.
\end{aligned}
\end{equation}
By the H\"older inequality and the Parabolic Sobolev embedding  
\begin{align*}
& \iint_{Q_{n+1}}\umkmnup^{\frac{m+1}m}\,dxd\tau\le
\Big[\sup_{t_j\le \tau\le t}\int_{B_j}[\umkmnup\zeta]^{\frac{m+1}m}(x,\tau)\,dx
\Big]^{\frac2n\frac{m+1}{qm}}\\
&\times\Big(\iint_{Q_j}|D\umkmnup|^2\zeta_j^2\,dxd\tau
+\iint_{Q_j}\umkmnup^2|D\zeta_j|^2\,dxd\tau\Big)^{\frac{m+1}{qm}}\\
&\times|Q_j|^{1-\frac{m+1}{qm}}\Big(\gm
\frac{2^{\frac{m+1}m j}}{k^{m+1}}
\frac1{|Q_j|}\iint_{Q_j}\umkmnp^{\frac{m+1}m}\,dxd\tau\Big)^{1-\frac{m+1}{qm}}
\end{align*}
where
\begin{equation*}
q=\frac{2(nm+m+1)}{nm}.
\end{equation*}
Now set
\begin{equation*}
Y_j=\frac1{|Q_j|}\iint_{Q_j}\umkmnp^{\frac{m+1}m}\,dxd\tau. 
\end{equation*}
Taking into account (\ref{Eq:B:4:4}), in terms of $Y_j$ 
the previous inequality becomes
\begin{equation*}
Y_{j+1}\le\gm\frac{b^j}{(1-
\sig)^{\frac{(m+1)(n+2)}{nqm}}k^{\frac{(m+1)(mq-m-1)}{qm}}}
\Big(\theta(2t)\Big)^{\frac{m+1}{qm}}
Y_j^{1+\frac{2(m+1)}{nqm}}, 
\end{equation*}
where 
\begin{equation*}
b=2^{\frac{3(m+1)}m(1+\frac{2(m+1)}{nqm})}.
\end{equation*}
Now $Y_j\to0$ as $j\to+\infty$, provided 
$k$ is chosen such that
\begin{equation*}
Y_o=\dashiint_{Q_o} u^{m+1}\,dxd\tau=\gm(1-\sig)^{\frac{n+2}2}
\Big(\theta(2t)\Big)^{\frac n2}k^{\frac{n(m-1)+2m+2}2}.
\end{equation*}
With this choice
\begin{equation}\label{Eq:B:4:5}
\begin{aligned}
M_\sig\le&\gm\frac{1}{(1-\sig)^{\frac{n+2}{n(m-1)+2m+2}}}
\Big(\frac1{\theta(2t)}\Big)^{\frac n{n(m-1)+2m+2}}\Big(\dashiint_{Q_o} u^{m+1}
\,dxd\tau\Big)^{\frac2{n(m-1)+2m+2}}. 
\end{aligned}
\end{equation}
Since $Q_o$ is intrinsic, up to a constant
\[
\theta(2t)=\left(\dashiint_{Q_o}u^{m+1}\,dxd\tau\right)^{\frac{1-m}{1+m}}.
\]
Hence, 
\begin{equation*}
M_\sig\le\gm\frac{1}{(1-\sig)^{\frac{n+2}{n(m-1)+2m+2}}}
\Big(\dashiint_{Q_o} u^{m+1}
\,dxd\tau\Big)^{\frac1{m+1}},
\end{equation*}
and the estimate does not change when we take into account the previous stipulation \eqref{Eq:B:4:3} about $k$, once more due to the intrinsic nature of $Q_o$. By the interpolation Lemma~6.1 of \cite{Giu03} 
we conclude.
\end{proof}
As a straightforward consequence, we obtain the following reverse H\"older inequality.
\begin{corollary}\label{RHI-for-u}
Let $u\ge0$ be a locally bounded, local, 
weak solution to \eqref{PME}--\eqref{PMS-eq:structure} 
 in $E_T$ with $\frac{(n-2)_+}{n+2}< m < 1$. Fix $t>0$, assume that $f\in L^\infty_{\loc}(E_T)$, the cylinder $B_{r(t)}(x_o)\times(t_o,t_o+t]\subset E_T$ is intrinsic with respect to $u$, and  
\[
\sup_{Q_{2t}}\left[r^2(2t)|f|^2\right]\le\frac c{t}\left[\theta(2t)\right]^{\frac{1+m}{1-m}}.
\]
for some $c>0$. 
Then there exists a positive constant $\gm$ depending only on 
the data $\datam$, such that for  $Q_{2t}:=B_{r(2t)}(x_o)\times(t_o-t,t_o+t]\subset E_T$, 
we have 
\begin{align*}
\left(\dashiint_{Q_{\frac32t}} 
u^{m+1}\,dxd\tau\right)^{\frac{1}{m+1}} &\le\gm
\left(\dashiint_{Q_{2t}} 
u^m\,dxd\tau\right)^{\frac{1}{m}}.
\end{align*}
\end{corollary}
\begin{remark}\label{rem:revH}
{\normalfont We point out that the derivation of the reverse H\"older estimate for $u^m$ is the only place where we make use of the $L^\infty$ bounds of the right-hand side. 
There are other, more rudimentary possibilities to derive a reverse H\"older estimate like the one in Corollary~\ref{RHI-for-u}, that would allow to consider vector-valued and signed solutions, weaker assumptions on $\A$ and the 
right-hand side $f$. Indeed, it can be achieved in three steps, by first applying a parabolic \Poincare inequality on $u^m$, then using the energy estimates, and finally an interpolation argument, as in Lemma~6.1 of \cite{Giu03}. We chose to limit ourselves to the scalar framework and nonnegative solutions, in order to allow a better understanding of the new techniques introduced in the paper, which we consider to be the priority in this paper.}
\end{remark}
%%%%%%%%%%%%%%%%%%%%%%%%%%%%%%%%%%%%%%%%%%%%%%%%%%%%%%%%%%%%%%%%%%%%
%%%%%%%%%%%%%%%%%%%%%%%%%%%%%%%%%%%%%%%%%
%%%%%%%%%%%%%%%%%%%%%%%%%%%%%%%%%%%%%%%%%
\section{Intrinsic Reverse H\"older Inequalities}\label{S:intrinsic}
In the next two subsections we will prove reverse 
H\"older inequalities in intrinsic cylinders. We will have to distinguish two different conditions, the so-called \emph{Degenerate} and \emph{Non-Degenerate} Regimes.
As a starting point we include the following lemma, which allows to switch mean values in time.
\begin{lemma}
\label{lem:time}
Take any cylinder $Q_{2s,2r}=(t_0-2s,t_0+2s)\times B_{2r}\subset E_T$,
and let $\tilde{\eta}$ be a constant-in-time cut off function, such that $\tilde\eta\in C^\infty_0((B_{2r}),[0,1])$, $\tilde\eta(x)\equiv 1$ for $x\in B_{r}$, and $\norm{D \eta}_\infty\leq\frac{c}{r}$. Then for $\displaystyle\eta=\frac{\tilde\eta}{\sqrt{\abs{ B_{r}}}}$ and for $t_0-2s\leq \sigma<\tau \leq t_0+2s$ we find
\begin{align}
\label{eq:time}
\abs{\Mean{u(\tau)}^{\eta^2}_{B_{2r}}-\Mean{u(\sigma)}^{\eta^2}_{B_{2r}}}\leq c \,s\bigg(\frac{1}{r}\dashiint_{Q_{2s,2r}}\abs{D  u^m}\,dxdt+\dashiint_{Q_{2s,2r}}\abs{f}\, dxdt\bigg).
\end{align}
Similarly, if we take $\displaystyle\eta=\frac{\tilde\eta}{\abs{ B_{r}}}$, we obtain
\begin{align}
\label{eq:time-bis}
\abs{{\Mean{u(\tau)}^{\eta}_{B_{2r}}}-{\Mean{u(\sigma)}^{\eta}_{B_{2r}}}}\leq c\,s\bigg( \frac{1}{r}\dashiint_{Q_{2s,2r}}\abs{D  u^m}\,dxdt+\dashiint_{Q_{2s,2r}}\abs{f}\, dxdt\bigg).
\end{align}
\end{lemma}
\begin{proof}
By the equation, we find that
\begin{align*}
\abs{\Mean{u(\tau)}^{\eta^2}_{B_{2r}}-\Mean{u(\sigma)}^{\eta^2}_{B_{2r}}}&=\abs{\int_\sigma^\tau\skp{\partial_t u(t)}{\eta^2}\, dt}\\
&=\Bigabs{c\int_\sigma^\tau\dashint_{B_{2r}}f(x,t)\,\eta^2-\A(x,t,u,Du^m)\cdot D  \eta\eta\,dxdt}
\\
&\leq \frac{cs}{r}\dashiint_{Q_{2s,2r}}\abs{D  u^m}\,dxdt+c\,s\dashiint_{Q_{2s,2r}} \abs{f}\,dxdt,
\end{align*}
where $c>1$ depends only on the data.
\end{proof}
\subsection{The Degenerate Regime}
In the following we consider an \emph{intrinsic} cylinder $Q_{s}^{\theta_{s}}\equiv Q_{s,\sqrt{s/\theta_s}}\equiv Q_{s,r(s)}\subset E_T$, and we assume that
\begin{align}\label{eq:deg1} 
\left[\dashiint_{Q_{s}^{\theta_{s}}}\abs{u^{m}-\Mean{u^{m}}_{Q_{s}^{\theta_{s}}}}^\frac{m+1}m\,dxdt\right]^{\frac{1}{m+1}}\geq \epsilon \left[\dashiint_{Q_{s}^{\theta_{s}}}u^{m+1}\,dxdt\right]^{\frac{1}{m+1}}
\end{align}  
for some $\epsilon>0$. We denote this condition as the \emph{Degenerate Regime}. Our aim is to obtain a reverse H\"older inequality for $|Du^m|$. We have the following proposition.

\begin{proposition}\label{Prop:rev-hold-deg}
Let $Q_{s}^{\theta_{s}}$ be one of the intrinsic cylinders constructed in Lemma~\ref{lem:scal} and assume it satisfies \eqref{eq:deg1}. Then, there exist $\vartheta_1\in (0,1)$ and a constant $c_\epsilon>1$ that depends only on $\epsilon, \vartheta_1$, and the data, such that
\begin{equation}\label{rev-hold-deg-I}
\begin{aligned}
\dashiint_{{Q_{s}^{\theta_{s}}}}\abs{Du^m}^2 dxdt
&\le c_\epsilon\bigg(\dashiint_{Q_{3s,3r(s)}}\abs{Du^m}^{2\vartheta_1} dxdt\bigg)^\frac{1}{\vartheta_1}
+c_\epsilon\sup_{Q_{3s,3r(s)}}[r^2(s)|f|^2].
\end{aligned}
\end{equation}
\end{proposition}
\begin{proof}
Without loss of generality, we can assume that
\[
\sup_{Q_{3s,3r(s)}}[r^2(s)|f|^2]\,\le c\frac{\theta_s^{\frac{m+1}{1-m}}}{s}
\]
for some positive constant $c$; otherwise, by Corollary~\ref{cor:deg-subint}, we immediately have that
\[
\dashiint_{Q_{s}^{\theta_{s}}}\abs{Du^m}^2 dxdt\le\gm\sup_{Q_{3s,3r(s)}}[r^2(s)|f|^2]
\]
and there is nothing to prove.

In order to conclude, it suffices to show that under the previous assumption on $f$ we have
\begin{equation}\label{rev-hold-deg}
\begin{aligned}
\dashiint_{Q_{s}^{\theta_{s}}}\abs{Du^m}^2 dxdt\leq c\frac{\theta_s^\frac{m+1}{1-m}}{s}\leq& 
c_\epsilon\bigg(\dashiint_{Q_{3s,3r(s)}}\abs{Du^m}^{2\vartheta_1} dxdt
\bigg)^{\frac{1}{\vartheta_1}}\\
&+c_\epsilon\,\dashiint_{Q_{3s,3r(s)}} r^2(s)|f|^2\,dxdt.
\end{aligned}
\end{equation}
First observe, that by Lemma~\ref{lm:intr-cons}, the cylinders $Q_{bs}^{\theta_{bs}}$ are intrinsic for $b\in [1,2]$. 
%\textcolor{red}{{\bf We have to be careful, because Lemma~\ref{lm:intr-cons} shows that intrinsic is preserved expanding $s$, whereas here we are claiming that intrinsic is preserved reducing $s$. It should be OK, but we need to prove it.}} 
By \eqref{scal:3} of Lemma~\ref{lem:scal} and the best constant property, we find that for all $b\in [1,2]$ 
\[
\theta_{bs}^\frac{m+1}{1-m}\leq \frac{c}{\epsilon}\dashiint_{Q_{bs}^{\theta_{bs}}}\abs{u^{m}-\Mean{u^{m}}_{Q_{bs}^{\theta_{2s}}}}^\frac{m+1}m\,dxdt.
\]
Moreover, \eqref{scal:7}--\eqref{scal:2} of Lemma~\ref{lem:scal} and the definition of $\theta_s=\frac{s}{(r(s))^2}$ imply
\begin{align*}
\dashiint_{Q_{2s,\sqrt{4s/\theta_s}}} u^{m+1}\, dxdt&=\dashiint_{Q_{2s,2r(s)}} u^{m+1}\, dxdt\leq c\,\dashiint_{Q_{\tilde{2}s,r(\tilde{2} s)}} u^{m+1}\, dxdt\\
&\leq c\,\left[\theta_{\tilde{2}s}\right]^{\frac{1+m}{1-m}}=c\,\left[\frac{\tilde{2}s}{(r(\tilde{2}s))^2}\right]^{\frac{1+m}{1-m}}\leq c\,\left[\frac{s}{(r(s))^2}\right]^{\frac{1+m}{1-m}}\leq c\,\left[\frac{\theta_{s}}2\right]^{\frac{1+m}{1-m}}
\end{align*}
Hence, the cylinder $Q_{2s,\sqrt{4s/\theta_s}}$ satisfies the assumption of Corollary~\ref{cor:deg-subint}; relying on it and  using the definition of $\theta_s=\frac{s}{(r(s))^2}$ again, yields
\begin{align*}
\frac{1}{s}\esssup_{t \in \Lambda_{s}} &\,\dashint_{B_{\sqrt{s/\theta_s}}} u^{m+1}\,dx
+ \dashiint_{Q_{s,\sqrt{s/\theta_s}}}  \abs{Du^{m}}^{2}\,dxdt \leq c\frac{\theta_{s}^\frac{m+1}{1-m}}{s}.
\end{align*}
This proves the first part of \eqref{rev-hold-deg}.
Now, we choose $\eta$ as in Lemma~\ref{lem:time}, assuming it has support in $B_{3r(s)}$ and it is constant in $B_{2r(s)}$, and
define the constant
\[
\lambda_o^m=\bigg(\dashint_{\Lambda_{2s}}\Mean{u(t)}^{\eta^2}_{B_{3r(s)}}\, dt\bigg)^{m}.
\]
We can estimate $\lambda_o$ from below using the reverse H\"older inequality of Corollary~\ref{RHI-for-u} and the 
intrinsic nature of $Q_s^{\theta_s}$ 
to find 
\begin{align*}
\lambda_o^m&\geq \bigg(\dashint_{\Lambda_{2s}}\frac1{\abs{B_{3r(s)}}}\int_{B_{2r(s)}}u\, dxdt\bigg)^{m}\geq c\,\bigg(\dashiint_{Q_{2s,2r(s)}}u\ dxdt\bigg)^m
\\
&\geq  c\,\bigg(\dashiint_{Q_{s,r(s)}}u^{m+1}\ dxdt\bigg)^\frac{m}{m+1}
= c\,\theta_{s}^\frac{m}{1-m}.
\end{align*}
%We estimate $\theta_s$ by using the best constant property, \eqref{meanchange} and \eqref{meanchange2}
%\begin{align*}
%\lambda_o^m&\geq \bigg(\dashint_{\Lambda_{2s}}\frac{\abs{B_{r(s)}}}{\norm{\eta^2}_1}\int_{B_{2r(s)}}u(t)\, dx\, dt\bigg)^{m}\geq c\bigg(\dashint_{Q_{2s,2r(s)}}u\ dz\bigg)^m
%\\
%&\geq  c\bigg(\dashint_{Q_{s,r(s)}}u^{m+1}\ dz\bigg)^\frac{m}{m+1}
%\geq c\theta_{s}^\frac{m}{1-m}.
%\end{align*}
We estimate $\theta_s$ by using the best constant property, \eqref{meanchange} and \eqref{meanchange2}. 
For some $\alpha\in(0,1)$ to be chosen, we have
\begin{align*}
&\theta_{s}^\frac{m+1}{1-m}=c\,\dashiint_{Q_s^{\theta_s}}u^{m+1}\,dxdt
\leq \frac{c}{\epsilon}\,\dashiint_{Q_{s}^{\theta_{s}}}\abs{u^{m}-\Mean{u^{m}}_{Q_{s}^{\theta_{s}}}}^\frac{m+1}m\,dxdt
\\
&\quad\leq 
\frac{c}{\epsilon}\bigg(\sup_{t\in \Lambda_{s}}\dashint_{B_{\sqrt{s/\theta_s}}}u^{m+1}\, dx\bigg)^{1-\alpha} \dashint_{\Lambda_s}\bigg(\dashint_{B_{\sqrt{s/\theta_{s}}}}\abs{u^{m}-\lambda_o^m}^\frac{m+1}m\,dx \bigg)^\alpha\, dt
\\
&\quad\leq 
\frac{c}{\epsilon}\bigg(\sup_{t\in \Lambda_{s}}\dashint_{B_{\sqrt{s/\theta_s}}}u^{m+1}\, dx\bigg)^{1-\alpha} \dashint_{\Lambda_s}\bigg(\dashint_{B_{3r(s)}}\abs{u^{m}-\Mean{u(t)}_{B_{3r(s)}}^m}^\frac{m+1}m\,dx 
\\
&\qquad +\Bigabs{\Big(\Mean{u(t)}^{\eta^2}_{B_{3r(s)}}\Big)^{m}-\lambda_o^m}^\frac{m+1}{m}\,\bigg)^\alpha dt.
\end{align*}
By the intrinsic nature of $Q_{2s,\sqrt{4s/\theta_s}}$, Corollary~\ref{cor:deg-subint} and Lemma~\ref{lm:intr-cons}, this implies that
\begin{align*}
&\Big(\frac{\theta_{s}^\frac{m+1}{1-m}}{s}\Big)^\alpha\leq\frac{c_\epsilon}{s^\alpha} \dashint_{\Lambda_s}\bigg(\dashint_{B_{3r(s)}}\abs{u^{m}-\Mean{u(t)}_{B_{3r(s)}}^m}^\frac{m+1}m\,dx \bigg)^\alpha\,dt
\\
&\qquad +\frac{c_\epsilon}{s^\alpha}\bigg(\dashint_{\Lambda_s} \Bigabs{\Big(\Mean{u(t)}^{\eta^2}_{B_{3r(s)}}\Big)^{m}-\lambda_o^m}^\frac{m+1}{m}\,dt\bigg)^\alpha=:(I)+(II)
\end{align*}
By the Sobolev-\Poincare inequality there exists $\vartheta_1\in (\frac12,1)$ such that
\begin{align*}
\dashint_{B_{3r(s)}}\abs{u^{m}-\Mean{u(t)}_{B_{3r(s)}}^m}^\frac{m+1}m\,dx
&\leq c\,\dashint_{B_{3r(s)}}\abs{u^{m}-\Mean{u^m(t)}_{B_{3r(s)}}}^\frac{m+1}m\,dx
\\
&\leq c(r(s))^\frac{m+1}{m}\bigg(\dashint_{B_{3r(s)}}\abs{D u^m}^{2\vartheta_1}\, dx\bigg)^\frac{m+1}{2\vartheta_1 m}.
\end{align*}
Now we choose
\[
\alpha=\frac{2\vartheta_1 m}{m+1},
\]
and we conclude that 
\begin{align*}
(I)&\leq \frac{c_\epsilon (r(s))^{2\vartheta_1}}{s^\alpha}\dashiint_{Q_{3s,3r(s)}}\abs{Du^m}^{2\vartheta_1}\,dxdt=\bigg(\frac{c_\epsilon (r(s))^{2}}{s^\frac{\alpha}{\vartheta_1}}\bigg(\dashiint_{Q_{3s,3r(s)}}\abs{Du^m}^{2\vartheta_1}\,dxdt\bigg)^\frac{1}{\vartheta_1}\bigg)^{\vartheta_1}.
%\\&\leq \bigg(\frac{c_\epsilon (r(s))^{2}}{s^\frac{2m}{m+1}}\bigg(\dashint_{Q_{3s,3r(s)}}\abs{Du^m}^{2\textcolor{red}{\vartheta_1}}\dz\bigg)^\frac{1}{\textcolor{red}{\vartheta_1}}\bigg)^\textcolor{red}{\vartheta_1}
\end{align*}
On the other hand, the Fundamental Theorem of Calculus, %\eqref{scal:-1}, 
Lemma~\ref{lem:time} and the fact that $r^2(s)\theta_s=s$ yield
\begin{align*}
\Bigabs{\big(\Mean{u(t)}_{B_{3r(s)}}^{\eta^2}\big)^{m}-\lambda_o^m}
&\sim  \big(\Mean{u(t)}_{B_{3r(s)}}^{\eta^2}+\lambda_o)^{m-1}\Bigabs{\Mean{u(t)}_{B_{3r(s)}}^{\eta^2}-\lambda_o}
\\
&\leq c\,\lambda_o^{m-1} \,\Bigabs{\dashint_{\Lambda_{2s}}[\Mean{u(t)}_{B_{3r(s)}}^{\eta^2}-\Mean{u(\tau)}_{B_{3r(s)}}^{\eta^2}]\, d\tau}
\\
&\leq \frac{c}{\theta_s}\,\dashint_{\Lambda_{2s}}\abs{\Mean{u(t)}_{B_{\sqrt{9s/\theta_{s}}}}^{\eta^2}-\Mean{u(\tau)}_{B_{\sqrt{9s/\theta_{s}}}}^{\eta^2}}\, d\tau
\\
&\leq  \frac{cs}{\theta_s}\bigg(\frac1{r(s)}\dashiint_{Q_{3s,3r(s)}}\abs{D  u^m}\, dxdt+\dashiint_{Q_{3s,3r(s)}}\abs{f}\, dxdt\bigg).
\end{align*}
This implies that
\begin{align*}
\Big(\frac{\theta_{s}^\frac{m+1}{1-m}}{s}\Big)^\frac{\alpha}{\vartheta_1}&\leq c(I)^\frac{1}{\vartheta_1}+ c(II)^\frac{1}{\vartheta_1}
\\
&\leq \frac{c_\epsilon (r(s))^{2}}{s^\frac{\alpha}{\vartheta_1}}\bigg(\dashiint_{Q_{3s,3r(s)}}\abs{Du^m}^{2\vartheta_1}\,dxdt\bigg)^\frac{1}{\vartheta_1}
\\
&\quad
+\frac{c_\epsilon}{s^\frac{\alpha}{\vartheta_1}}\bigg(\frac{s}{r(s)\theta_s}\bigg(\dashiint_{Q_{3s,3r(s)}}\abs{D u^m}\,dxdt+\dashiint_{Q_{3s,3r(s)}} r(s)\abs{f}\,dxdt\bigg)\bigg)^{\frac{m+1}{m}\frac{\alpha}{\vartheta_1}}\\
&= \frac{c_\epsilon s}{s^\frac{\alpha}{\vartheta_1}\theta_s}\bigg(\dashiint_{Q_{3s,3r(s)}}\abs{Du^m}^{2\vartheta_1}\,dxdt\bigg)^\frac{1}{\vartheta_1}
\\
&\quad
+\frac{c_\epsilon}{s^\frac{\alpha}{\vartheta_1}}\Big(\frac{s}{r(s)\theta_s}\Big)^2\bigg(\dashiint_{Q_{3s,3r(s)}}\abs{D  u^m}\,dxdt+\dashiint_{Q_{3s,3r(s)}} r(s)\abs{f}\, dxdt\bigg)^{2}
\\
&= \frac{c_\epsilon s}{s^\frac{\alpha}{\vartheta_1}\theta_s}\bigg(\dashiint_{Q_{3s,3r(s)}}\abs{Du^m}^{2\vartheta_1}\,dxdt\bigg)^\frac{1}{\vartheta_1}
\\
&\quad
+\frac{c_\epsilon}{s^{\frac{\alpha}{\vartheta_1}}}\frac{s}{\theta_s}\bigg(\dashiint_{Q_{3s,3r(s)}}\abs{D  u^m}\,dxdt+\dashiint_{Q_{3s,3r(s)}} r(s)\abs{f}\,dxdt\bigg)^{2}.
\end{align*}
By Jensen's inequality, this implies that
\begin{equation}
\begin{aligned}
\frac{\theta_s^\frac{m+1}{1-m}}{s}=\theta_s^\frac{2m}{1-m}\frac{\theta_s}{s}=\Big(\frac{\theta_{s}^\frac{m+1}{1-m}}{s}\Big)^\frac{\alpha}{\vartheta_1}\frac{s^{\frac{\alpha}{\vartheta_1}}\theta_s}{s}\leq &c_\epsilon\bigg[\dashiint_{Q_{3s,3r(s)}}\abs{D  u^m}^{2\vartheta_1}\,dxdt\bigg]^\frac{1}{\vartheta_1}\\
&+c_\epsilon\,\dashiint_{Q_{3s,3r(s)}} r(s)^{2}\abs{f}^{2}\,dxdt.
\end{aligned}
\end{equation}
\end{proof}
\subsection{The Non-Degenerate Regime}
In the following we assume that the opposite of $\eqref{eq:deg1}$ holds. Namely,
we consider an intrinsic cylinder $Q_{2s}^{\theta_{2s}}\equiv Q_{2s,\sqrt{2s/\theta_{2s}}}\equiv Q_{2s,r(2s)}\subset E_T$,
and we assume that
\begin{align}
\label{eq:ndeg} \bigg(\dashiint_{Q_{2s}^{\theta_{2s}}}\abs{u^m-\Mean{u^m}_{Q_{2s}^{\theta_{2s}}}}^\frac{m+1}{m}dxdt\bigg)^\frac1{m+1}\leq \epsilon \bigg(\dashiint_{Q_{2s}^{\theta_{2s}}}u^{m+1}dxdt\bigg)^\frac1{m+1}.
\end{align}
We denote this condition the \emph{Non-Degenerate Regime}; heuristically, it implies that $u$ is close to a solution of the linear heat equation.
%Due to Lemma~\ref{lem that for all $a\in [2,4]$
%\begin{align}
%\label{eq:deg} \bigg(\dashint_{Q_{a\theta\rho^2,a\rho}}\abs{u-\Mean{u}_{Q_{a\theta\rho^2,a\rho}}}^{m+1}\bigg)^\frac1{m+1}\geq \epsilon \bigg(\dashint_{Q_{a\theta\rho^2,a\rho}}u^{m+1}\bigg)^\frac1{m+1}.
%\end{align}
\begin{proposition}\label{Prop:rev-hol-non-deg}
Let $Q_{2s}^{\theta_{2s}}$ be one of the intrinsic cylinders constructed in Lemma~\ref{lem:scal} and assume it satisfies \eqref{eq:ndeg}. Then, there exist $\vartheta_2\in (0,1)$ and a constant $c_\epsilon$ that depends just on $\epsilon,\ \vartheta_2$ and the data, such that
\begin{equation*}
\begin{aligned}
\dashiint_{Q_{s}^{\theta_{2s}/2}}\abs{Du^m}^2\,dxdt&\leq
c_\epsilon\bigg(\dashiint_{Q_{2s}^{\theta_{2s}}}\abs{Du^m}^{2\vartheta_2}\,dxdt\bigg)^\frac{1}{\vartheta_2}+c_\epsilon\dashiint_{Q_{2s}^{\theta_{2s}}}r^2(2s)|f|^2\,dxdt
\\
&\le c_\epsilon\bigg(\dashiint_{Q_{2s}^{\theta_{2s}}}\abs{Du^m}^{2\vartheta_2}\,dxdt\bigg)^\frac{1}{\vartheta_2}+c_\epsilon\,\sup_{Q_{2s}^{\theta_{2s}}}[r^2(2s)|f|^2].
\end{aligned}
\end{equation*}
\end{proposition}
\begin{proof}
We assume $Q_{2s}^{\theta_{2s}}$ is an intrinsic cylinder, where the non-degenerate regime holds. 
In this context, we will not make use of the reverse H\"older inequality of Corollary~\ref{RHI-for-u}. 

We may suppose that for some $\delta>0$ sufficiently small
\[
(r(2s))^2\dashiint_{Q_{2s}^{\theta_{2s}}}\abs{f}^2\,dxdt\le \delta\dashiint_{Q_{s}^{\theta_{2s}/2}}\abs{Du^m}^2\,dxdt;
\]
indeed, otherwise, there is nothing to prove.

Fix $\sigma \in [s,2s]$: by Lemma~\ref{lem:scal?} we know that $\theta_\sigma\sim \theta_{2s}$ and $r(\sigma)\sim r(2s)$. 
For the sake of simplicity, in the following we write 
\[
r=\frac{r(2s)}{2},\ \ \text{ and }\ \ \theta=\theta_{2s}. 
\]
Without further notice, we use the fact that $\theta r^2 = \frac{s}{2}$.
Moreover, by Lemma~\ref{lem:trick} we find 
\begin{align*}
\bigg(\dashiint_{Q_{2s,2r}}\abs{u^m-\Mean{u}_{Q_{2s,2r}}^m}^\frac{m+1}{m}\,dxdt\bigg)^\frac1{m+1}
\sim& 
\bigg(\dashiint_{Q_{2s,2r}}\abs{u^m-\Mean{u^m}_{Q_{2s,2r}}}^{\frac{m+1}{m}}\,dxdt\bigg)^\frac1{m+1}
\\
\leq&
c\epsilon\,\bigg(\dashiint_{Q_{2s,2r}}u^{m+1}dxdt\bigg)^\frac1{m+1}
\\
\leq&
c\epsilon \bigg(\dashiint_{Q_{2s,2r}}\abs{u^m-\Mean{u^m}_{Q_{2s,2r}}}^\frac{m+1}{m}dxdt\bigg)^\frac1{m+1}\\
&+c\epsilon\Mean{u^m}_{Q_{2s,2r}}^\frac1m 
\\
\leq&
c\epsilon^2 \bigg(\dashiint_{Q_{2s,2r}}u^{m+1}\,dxdt\bigg)^\frac1{m+1}+c\epsilon\Mean{u^m}_{Q_{2s,2r}}^\frac1m 
\end{align*}
which implies
\[
\Mean{u^m}_{Q_{2s,2r}}^\frac1m\leq \frac{c\epsilon}{1-c\epsilon^2}\bigg(\dashiint_{Q_{2s,2r}}u^{m+1}\,dxdt\bigg)^\frac1{m+1}, 
\]
and 
\begin{align*}
\bigg(\dashiint_{Q_{2s,2r}}\abs{u^m-\Mean{u}_{Q_{2s,2r}}^m}^\frac{m+1}{m}\,dxdt\bigg)^\frac1{m+1}
&\sim 
\bigg(\dashiint_{Q_{2s,2r}}\abs{u^m-\Mean{u^m}_{Q_{2s,2r}}}^\frac{m+1}{m}\,dxdt\bigg)^\frac1{m+1}
\\
&\leq 
\frac{c\epsilon}{1-c\epsilon}\Mean{u^m}_{Q_{2s,2r}}^\frac1m\leq  \frac{c\epsilon}{1-c\epsilon}\Mean{u}_{Q_{2s,2r}}.
\end{align*}
Next, let $\tilde{\eta}$ to be a constant-in-time cut off function, such that $\tilde\eta\in C^\infty_0((B_{2r}),[0,1])$, $\tilde\eta(x)\equiv 1$ for $x\in B_{r}$, and $\abs{D \tilde\eta}_\infty\leq\frac{c}{r}$. 
We set $\displaystyle \eta=\frac{\tilde\eta}{\abs{B_{r}}}$, and define
\[
\lambda=\Mean{u}^{\eta}, \lambda(t)=\Mean{u(t)}^{{\eta}}\text{ and }e^m=\Mean{u^m}^{\eta},e^m(t)=\Mean{(u(t))^m}^{\eta}.
\]
By Lemma~\ref{lm:intr-cons} %and Corollary~\ref{RHI-for-u} 
we find that 
\[
\Mean{u^{m+1}}_{Q_{2s,2r}}^\frac1{m+1}\sim \Mean{u^{m+1}}_{Q_{s,r}}^\frac1{m+1}\ \ \text{ and }\ \ \Mean{u}_{Q_{2s,2r}}\sim \Mean{u}_{Q_{s,r}},
\]
which directly implies that  
\[
\Mean{u^{m+1}}_{Q_{2s,2r}}\sim  \Mean{u^{m+1}}^{\eta},\ \ \text{ and }\ \ \Mean{u}_{Q_{2s,2r}}\sim 
\Mean{u}^{\eta}.
\]
We conclude that
\begin{align*}
\lambda \sim \Mean{u^m}_{Q_{as,ar}}^\frac1m\sim \Mean{u^{m+1}}_{Q_{2s,2r}}^\frac1{m+1}\sim \theta_{2s}^\frac{1}{1-m}%\sim \Mean{u}_{Q_{as,ar}}
\end{align*}
 for all $a\in [\frac12,1]$.
For any $1/2\leq a <b\leq 1$ and $\theta_{2s}\sim\lambda^{1-m}>0$ we find by the energy estimate of Lemma~\ref{lem:osc-energy} with $c=\lambda$
\begin{align*}
\begin{aligned}
I(a)+II(a):=&
\frac{\lambda^{m-1}}{r^2}\esssup_{t \in \Lambda_{as}} 
\bigg(\dashint_{B_{ar}} \abs{u^{m}-\lambda^{m}}^\frac{m+1}{m}\,dx\\
& +\dashint_{B_{ar}} \abs{u^{m}-\lambda^{m}}\abs{u-\lambda}\,dx\bigg)
+\dashiint_{Q_{as,ar}}  \abs{Du^{m}}^{2}\,dxdt \\
\le& \frac{c\lambda^{m-1}}{r^2(b-a)} \dashiint_{Q_{bs,br}} \abs{u^{m}-\lambda^{m}} \abs{u-\lambda}\,dxdt\\  
&+ \frac{c}{(b-a)^2r^2} \dashiint_{Q_{bs,br}}\abs{u^{m}-\lambda^{m}}^2\,dxdt+c\,\dashiint_{Q_{bs,br}}[r^2|f|^2]\,dxdt\\
\le& \frac{c\lambda^{m-1}}{r^2(b-a)} \dashiint_{Q_{bs,br}} \abs{u^{m}-\lambda^{m}} \abs{u-\lambda}\,dxdt\\  
&+ \frac{c}{(b-a)^2r^2} \dashiint_{Q_{bs,br}}\abs{u^{m}-\lambda^{m}}^2\,dxdt
+c\delta\,\dashiint_{Q_{as,ar}}  \abs{Du^{m}}^{2}\,dxdt
%+
%\textcolor{blue}{\frac{c}{(b-a)^2r^2}\,\lambda^{2m}}. 
\end{aligned}
\end{align*}
Hence, by fixing $\delta$ sufficiently small, by absorption we find that 
\begin{align*}
\begin{aligned}
I(a)+II(a)
\le& \frac{c\lambda^{m-1}}{r^2(b-a)} \dashiint_{Q_{bs,br}} \abs{u^{m}-\lambda^{m}} \abs{u-\lambda}\,dxdt\\
&+ \frac{c}{(b-a)^2r^2} \dashiint_{Q_{bs,br}}\abs{u^{m}-\lambda^{m}}^2\,dxdt
%+\seb{c\delta}\dashiint_{Q_{as,ar}}  \abs{Du^{m}}^{2}\,dxdt
%+
%\textcolor{blue}{\frac{c}{(b-a)^2r^2}\,\lambda^{2m}}. 
\end{aligned}
\end{align*}
We continue using Young's inequality and the fact that
\[
\abs{u^m-\lambda^m}^2\sim (u+\lambda)^{2m-2}\abs{u-\lambda}^2 \leq \lambda^{m-1}\abs{u-\lambda}^2 (u+\lambda)^{m-1}\sim \lambda^{m-1}\abs{u-\lambda}\abs{u^m-\lambda^{m}},
\]
and find
\begin{align*}
\begin{aligned}
I(a)+II(a)&\leq \frac{c}{(b-a)^2}\esssup_{t \in \Lambda_{bs}}\bigg(\frac{\lambda^{m-1}}{r^2}\dashint_{B_{br}} \abs{u^{m}-\lambda^{m}} \abs{u-\lambda}\,dx\bigg)^{1-\alpha}
\\
&\quad \times \dashint_{\Lambda_{bs}}\bigg[\bigg(\frac{\lambda^{m-1}}{r^2}\dashint_{B_{br}} \abs{u^{m}-\lambda^{m}} \abs{u-\lambda}\,dx\bigg)^\alpha + \bigg(\frac1{r^2}\dashint_{B_{br}} \abs{u^{m}-\lambda^{m}}^2\,dx\bigg)^\alpha\bigg]\,dt  
\\
&=(I(b))^{1-\alpha} \frac{c}{(b-a)^2}\dashint_{\Lambda_{bs}}\bigg(\bigg(\frac{\lambda^{m-1}}{r^2}\dashint_{B_{br}} \abs{u^{m}-\lambda^{m}} \abs{u-\lambda}\,dx\bigg)^\alpha\\
&\quad + \bigg(\frac1{r^2}\dashint_{B_{br}} \abs{u^{m}-\lambda^{m}}^2\,dx\bigg)^\alpha\bigg) dt 
\\
&\leq \frac12 (I(b))+\frac{c}{(b-a)^\frac{2}{\alpha}} \bigg(\dashint_{\Lambda_{bs}}\bigg(\bigg(\frac{\lambda^{m-1}}{r^2}\dashint_{B_{br}} \abs{u^{m}-\lambda^{m}} \abs{u-\lambda}\,dx\bigg)^\alpha\\
&\quad + \bigg(\frac1{r^2}\,\dashint_{B_{br}} \abs{u^{m}-\lambda^{m}}^2\,dx\bigg)^\alpha\bigg) dt\bigg)^\frac{1}{\alpha}
\\
&=:    \frac12 I(b)+\frac{c}{(b-a)^\frac{2}{\alpha}}(III).
\end{aligned}
\end{align*}
Since we can interpolate the essential supremum to the left-hand side, we are left to estimate $(III)$.
Observe that
\begin{align*}
\dashint_{B_{2r}} \abs{u^{m}-\lambda^{m}} \abs{u-\lambda}\eta&\,dx
=\dashint_{B_{2r}} \chi_{\set{u\geq 2\lambda}}\abs{u^{m}-\lambda^{m}} \abs{u-\lambda}\eta\,dx\\
& + \dashint_{B_{2r}} \chi_{\set{u < 2\lambda}}\abs{u^{m}-\lambda^{m}} \abs{u-\lambda}\eta\,dx
\\
\leq&\dashint_{B_{2r}} \chi_{\set{u\geq 2\lambda}}u^{m+1}\eta\,dx\\ 
&+ \lambda^{1-m}\dashint_{B_{2r}} \chi_{\set{u < 2\lambda}}\abs{u^{m}-\lambda^{m}} \lambda^{m-1}\abs{u-\lambda}\eta\,dx
\\
\leq& c\bigg(\dashint_{B_{2r}} u^{m+1}\, dx\bigg)^\frac{1-m}{m+1}\bigg(\dashint_{B_{2r}} \chi_{\set{u\geq 2\lambda}}\abs{u^{m}-\lambda^m}^\frac{m+1}{m}\eta\,dx\bigg)^\frac{2m}{m+1} 
\\
&+c\,\lambda^{1-m}\dashint_{B_{2r}} \chi_{\set{u < 2\lambda}}\abs{u^{m}-\lambda^{m}} (\lambda+u)^{m-1}\abs{u-\lambda}\eta\,dx
\\
\leq&c\bigg(\lambda^{1-m}+\bigg(\dashint_{B_{2r}} u^{m+1}\, dx\bigg)^\frac{1-m}{m+1}\bigg)\bigg(\dashint_{B_{2r}} \abs{u^{m}-\lambda^m}^\frac{m+1}{m}\eta\,dx\bigg)^\frac{2m}{m+1}. 
\end{align*}
This allows to conclude via H\"older's inequality (using $\frac{1-m}{m+1}+\frac{2m}{m+1}=1$) and Jensen's inequality  that
\begin{align*}
(III)^\alpha&\leq c
\dashint_{\Lambda_{s}}\bigg(\bigg(\frac{\lambda^{m-1}}{r^2}\dashint_{B_{2r}} \abs{u^{m}-\lambda^{m}} \abs{u-\lambda}\eta\,dx\bigg)^\alpha + \bigg(\frac1{r^2}\dashint_{B_{2r}} \abs{u^{m}-\lambda^{m}}^2\eta \,dx\bigg)^\alpha\bigg) dt 
\\
&\quad\leq 
c  \bigg(\dashint_{\Lambda_{s}} \bigg(\frac1{r^{\frac{m+1}m}}\dashint_{B_{2r}} \abs{u^{m}-\lambda^{m}}^\frac{m+1}{m}\eta\,dx\bigg)^\alpha dt\bigg)^\frac{2m}{m+1} 
\end{align*}
Next we may subtract the space mean values and find by Lemma~\ref{lem:trick}
\begin{align*}
\bigg(\dashint_{B_{2r}} \abs{u^{m}-\lambda^m}^\frac{m+1}{m}\eta\,dx\bigg)^\frac{m}{m+1}
&\leq 
c\bigg(\dashint_{B_{2r}} \abs{u^{m}-\lambda^m(t)}^\frac{m+1}{m}\eta\,dx\bigg)^\frac{m}{m+1}+c\abs{\lambda^m(t)-\lambda^m}
\\
&\leq 
c\bigg(\dashint_{B_{2r}} \abs{u^{m}-e^m(t)}^\frac{m+1}{m}\eta\,dx\bigg)^\frac{m}{m+1}+c\abs{\lambda^m(t)-\lambda^m}
\\
&=(i)+(ii).
\end{align*}
Using \eqref{meanchange2} yields
\begin{align*}
\frac{(i)^2}{r^2}\leq \frac{c}{r^2}\bigg(\dashint_{B_{2r}} \abs{u^{m}-e^m(t)}^\frac{m+1}{m}\eta\,dx\bigg)^\frac{2m}{m+1}
\leq c\bigg(\dashint_{B_{2r}} \abs{D  u^m}^{2\vartheta_2}\,dx\bigg)^\frac{1}{\vartheta_2}.
\end{align*}
We further estimate using Lemma~\ref{lem:time}, to get that
\begin{align*}
\frac{(ii)}{r}&=\frac{1}{r}\abs{\lambda^m-\lambda^m(t)}\sim \frac{1}{r}(\lambda+\lambda(t))^{m-1}\abs{\lambda(t)-\lambda}\\
&\leq  \frac{c\lambda^{m-1}s}{r^2(s)}\dashiint_{Q_{2s,2r}}\abs{D  u^m}\,dxdt\leq c\dashiint_{Q_{2s,2r}}\abs{D  u^m}\,dxdt.
\end{align*}
By choosing $\alpha =\frac{2m\,\vartheta_2}{m+1}$ we find
\begin{align*}
(III)&\leq c  \bigg(\dashint_{\Lambda_{s}} \bigg(\dashiint_{Q_{2s,2r}}\abs{D  u^m}\,dxdt\bigg)^{\alpha}+\dashint_{B_{2r}} \abs{Du^m}^{2\vartheta_2}\,dxdt\bigg)^\frac{ 2m}{\alpha(m+1)}
\\
&\leq c \bigg(\dashiint_{Q_{2s,2r}}\abs{D  u^m}^{2\vartheta_2}\,dxdt\bigg)^\frac{1}{\vartheta_2}.
\end{align*}
Hence, the result follows by the interpolation result of \cite[Lemma~6.1]{Giu03}.
\end{proof}

\section{Higher Integrability}\label{sec:int}
\subsection{Covering}
We can finally come to the core of our argument: in the following we will build a proper covering with respect to $u$ for the level sets of the function 
$\abs{D u^m}^2$.

We assume that we have a sub-intrinsic cylinder
\begin{align}\label{startcylinder}
C\theta_o\geq\bigg(\dashiint_{Q_{2\theta_o R^2,2R}}u^{m+1}\,dxdt\bigg)^\frac{1-m}{m+1}.
\end{align}
We start by scaling everything to the cube $Q_{2,2}$. This can be done by introducing $\lambda_o=\theta_o^\frac{1}{m-1}$. Then we define $\tilde{u}(y,s)=\lambda_o u(R y,\lambda_o^{m-1}R^2 s)$. For this scaled solution we find
\begin{align}
C\geq \lambda_o^{m+1}\dashiint_{Q_{2\lambda_o^{m-1}R^2,2R}}u^{m+1}\,dxdt =\dashiint_{Q_{2,2}}{\tilde u}^{m+1}\,dyds.
\end{align}
Moreover, $\tilde{u}$ is a weak solution in $Q_{2,2}$ to
\[
\tilde{u}_{s} -  \div \tilde{\A}(y,s,\tilde{u},D\tilde{u}) = \tilde{f} 
\]
with right-hand side $\tilde{f}(y,s)=\lambda_o^{m} R^2 f(R y,\lambda_o^{m-1}R^2s)$, and $\tilde{\A}$ which satisfies structure conditions analogous to \eqref{PMS-eq:structure}.

By \eqref{PMS-eq: energy modified}, the previous scaling and \eqref{startcylinder}, we find that
\begin{align*}
\lambda_o^{2m}R^2\dashiint_{Q_{\lambda_o^{m-1}R^2,2R}}\abs{D u^m}^2\,dxdt&=\dashiint_{Q_{1,2}}\abs{D \tilde{u}^m}^2\,dyds\\
&\leq c\,\left[\dashiint_{Q_{2,2}}\tilde{u}^{m+1}\,dyds\right]^{\frac{2m}{m+1}}+c\,\dashiint_{Q_{2,2}}\tilde{f}^2\,dyds \leq C(\tilde f).
\end{align*}
We introduce the notation 
  \begin{align}
  F=\abs{D \tilde{u}^m}^{2}\chi_{Q_{2,2}}.%\text{ d } G= u^{(m+1)q}.
  \end{align}
	We fix 
\begin{align}\label{b}	
\begin{aligned}
&\hat{b}\in\Big(0,\min\set{{(n+2)(m+1)-2n},\frac12}\Big)\text{ and }\beta=(1-2\hat{b})\in (0,1),\\
&\text{such that }1<\frac{1}{1-2\hat{b}}=\frac{1}{\beta}<\frac{m+1}{1-m}.
\end{aligned}
\end{align}
Let $\gamma_1$ be the constant in Corollary~\ref{cor:deg-subint} for $K=1$, and $\tilde{\delta}\in(0,\frac12)$ a parameter which is to be chosen later. Finally, let $c_1>1$ be fixed by  \eqref{scal?:2} of Lemma~\ref{lem:scal?}.
Accordingly, we fix $c_o,c_2\in (2,\infty)$ depending on $\tilde{\delta}$ and the data alone, such that
\begin{align}
\label{eq:c1}
\gamma_1 \left(\frac{2}{c_o}\right)^{\frac{m+1}{1-m}\beta-1} =\tilde{\delta}\text{ and } c_2= 2c_1\tilde{3}c_o,\,\text{ with }\,\tilde3=3^{1/\hat b}.
\end{align}

We will apply Lemma~\ref{lem:scal} with respect to $\tilde{u}$ and this fixed choice of $\hat{b}$ on the sub-intrinsic initial cylinder $Q_{2,2}$. To simplify the notation, for $y\in Q_{1,1}$ and $s\in (0,1)$, 
we denote the sub-intrinsic cube $\displaystyle Q_{s,r(s,y)}(y)$ defined in Lemma~\ref{lem:scal} with $\displaystyle Q(s,y)$. Moreover, we denote again with $x$ and $t$ the new variables $y$ and $s$.
\begin{lemma}
\label{lem:bigcubes}
Let $c_2$ be fixed by \eqref{eq:c1}.
For $\frac12\leq a<b\leq 1$, there exists a parameter $\mu_{a,b}$, defined by
\[
\mu_{a,b}=\frac{C(\tilde f,\tilde\delta)}{\abs{b-a}^{\tau}},
\]
with $\tau$ depending only on the data, such that if for $z\in Q_{a,a}$, the intersection $Q(c_2 s,z)\cap Q_{b,b}^c$ is not empty, then
\[
\dashiint_{Q(s,z)}F\,dxdt\leq \mu_{a,b}.
\]
Here $\displaystyle C(\tilde f,\tilde{\delta})=c\,\dashiint_{Q_{2,2}}\tilde{f}^{2}\,dxdt+c\,\left[\dashiint_{Q_{2,2}}\tilde{u}^{m+1}\,dxdt\right]^{\frac{2m}{m+1}}$, with $c>0$ depending only on the data, and $\tilde{\delta}$.
\end{lemma} 
\begin{proof}
If $Q(c_2s,z)\cap Q_{b,b}^c\neq \emptyset$, then either
$c_2s/2>(b-a)$, or $r(c_2s,z)>(b-a)$. In the latter case, by \eqref{scal:0} of Lemma~\ref{lem:scal} (using the fact that $S=1$ and $R(z)\sim1$ here), we have that $(b-a)\leq c_2^{\hat{b}}s^{\hat{b}}$ and by \eqref{scal:3} that $(b-a)\leq c_2^{\hat{a}}r(s,z)$. 

\noindent In case $c_2s/2>(b-a)$, we find that
$r(s,z)\geq r(\frac{b-a}{c_2},z)\geq c\abs{b-a}^{\hat{a}}R(z)\geq c\abs{b-a}^{\hat{a}}$.
 Therefore, \eqref{scal:3} of Lemma~\ref{lem:scal} implies for $\tau= \max\set{\hat{a},1}n+\frac{1}{\hat{b}}$, that
 \begin{align}
\label{eq:subintr}
\dashiint_{Q(s,z)} F\,dxdt\leq \frac{1}{\abs{b-a}^{\tau}}\dashiint_{Q_{2,2}} F\,dxdt\leq \frac{C(\tilde f,\tilde\delta)}{\abs{b-a}^{\tau}}.
 \end{align}
%where $\theta_{1,z}$ is defined via Lemma~\ref{lem:scal}. 
\end{proof}
Now, for a locally integrable function $g$, we introduce the related intrinsic maximal function
\[
\Mi(g)(z)=\sup_{Q(s,y)\ni z,\ y\in Q_{1,1}}\dashiint_{Q(s,y)}\abs{g}\,dxdt,
\]
and we define 
\begin{align}
M^*(g)(t,x):=\sup_{t\in I\subset (-2,2),\ x\in B\subset B_2}\dashint_I\dashint_B\abs{g}\,dxdt.
\end{align}
Notice that we have 
\begin{align}
\abs{g}\leq\Mi(g)\leq M^*(g)\qquad \text{ a.e.};
\end{align}
therefore, $\Mi$ is continuous from $L^q\to L^q$, whenever $g\in L^q$. 

We define the level sets of $F$ by 
\begin{align}
O_\lambda :=\set{\Mi(F)>\lambda}.
\end{align} 
	
The next proposition is the core of the proof of the higher integrability. It constructs a covering, which allows to exploit the reverse H\"older estimates of the previous section in a suitable way. It is a covering of Calderon-Zygmund type for $F$, build using cylinders scaled with respect to $\tilde u$. 

The fact that the scaling is done with respect to $\tilde u$, and not with respect to the function whose level sets are covered, i.e. $F$, makes things quite delicate. In this context, this seems the right way to proceed, instead of relying on the by-now standard approach of parabolic intrinsic Calderon-Zygmund covering, originally introduced by Kinnunen \& Lewis~\cite{Kinnunen:2000}.   
	
\begin{proposition}\label{lem:czcover}
	%There are constants $c^*\in(0,1)$ and $c^{**}>1$, such that
	For $\tilde{\delta}$ fixed in \eqref{eq:delta} below, and 
	$\frac12\leq a<b\leq 1$, let $\mu_{a,b}$ be the quantity introduced in Lemma~\ref{lem:bigcubes}. Let $q=\min{\set{\vartheta_1,\vartheta_2}}\in (0,1)$ where $\vartheta_1$, $\vartheta_2$ are the exponents defined via the reverse H\"older estimates of Propositions~\ref{Prop:rev-hold-deg} and \ref{Prop:rev-hol-non-deg}.
	
For every $\lambda>\mu_{a,b}$, and every $z\in O_\lambda\cap Q_{a,a}$, there exist 
$Q_z\subset Q_z^*\subset Q_z^{**}\subset Q_{b,b}$, which satisfy the following properties.
\begin{enumerate}
\item[(i)] $\abs{Q_z},\abs{Q_z^*},\abs{Q_z^{**}}$ are of comparable size.
\item[(ii)]
 For any $y\in Q_z^*$ there exist constants $c$ and $C$ just depending on the data such that
\begin{equation}\label{eq:ii}
\lambda\leq {c}\dashiint_{Q_{z}}F\,dxdt\leq C\bigg(\dashiint_{Q_{z}^*}F^q\,dxdt \bigg)^\frac1q + CM^*(\tilde{f^2})(y).
\end{equation}
\item[(iii)] $\displaystyle\dashiint_{Q_z^{**}} F\,dxdt\leq 2 \lambda$.
%\item The sub-cylinders $Q_i^{*}:=Q_{2s(y_i,c_o{r_i}),2c_or_i}(y_i)\subset Q_i^{**}\subset Q_{b,b}$ are disjoint. \textcolor{red}{\bf We still need to define $c_o$.} 
%
%\item The sub-cylinders $Q_i:=Q_{s(y_i,{r_i}),r_i}(y_i)$ satisfy
%% $Q_i^*=Q_{s(y_i,c_1r_i),c^*r_i}(y_i)$ are disjoint. And 
%
%for all $Q_{s(\rho,\xi),\rho}(\xi)\supset Q_{s(r,y),r}(y)$.
\end{enumerate}
Moreover, the set $O_\lambda\cap Q_{a,a}$ can be covered by a family of cylinders 
$Q_i^{**}:=Q_{z_i}^{**}\subset Q_{b,b}$, such that the cylinders of the family $Q_i^*$ are pairwise disjoint.
\end{proposition}
\begin{proof}
We fix $\lambda>\mu_{a,b}$, and for $z\in O_\lambda\cap Q_{a,a}$ we choose $y_z$ and $s_z$, such that $Q(s_z,y_z)\ni z$ and
\begin{align}\label{eq:lambda}
\lambda < \dashiint_{Q(s_z,y_z)}F\,dxdt\ \ \text{ and }\ \ \dashiint_{Q(\sigma,\xi)}F\,dxdt\leq 2\lambda
\end{align}
for all $Q{(\sigma,\xi)}\supset Q(s_z,y_z)$; such a cylinder certainly exists by the very definition of the set $O_\lambda$. 
In the following notation, for $c\in \setR^+$, we let $\tilde{c}=c^{1/\hat b}$, with $\hat b$ as in \eqref{b}.
For $z\in O_\lambda\cap Q_{a,a}$, we will carefully choose cylinders according to the following table . Here we rely on \eqref{eq:c1}; notice that we are assuming $c_o>2$.
\vskip.2truecm
\begin{easylist}
\ListProperties(Hide=100, Hang=true, Progressive=3ex, Style*=$\bullet$ ,
Style2*=$\circ$ ,Style3*=\tiny$\blacksquare$ )
\item {\bf Case 1}: There exists $\sigma\in [2s_z,c_o s_z]$ such that  $Q(\sigma,y_z)$ is intrinsic and the degenerate alternative holds in $Q(\sigma,y_z)$. Then we let\\
\begin{itemize}
\item $Q_z=Q(\sigma, y_z),\qquad Q_z^{*}=Q(\tilde{3}\sigma,y_z),\qquad Q_z^{**}=Q(2c_1\tilde{3}\sigma, y_z)$\\
\end{itemize}
\item {\bf Case 2}:  There exists $\sigma\in [2s_z,c_o s_z]$ such that  $Q(\sigma,y_z)$ is intrinsic and the non-degenerate alternative holds in $Q(\sigma,y_z)$. Then we let\\
\begin{itemize}
\item $Q_z=\frac12Q(\sigma, y_z),\qquad Q_z^{*}=Q(\sigma,y_z),\qquad Q_z^{**}=Q(2c_1\sigma, y_z)$\\
\end{itemize}
\item {\bf Case 3}: There exists no $\sigma\in [2s_z,c_o s_z]$ such that  $Q(\sigma,y_z)$ is intrinsic. Then we let\\
\begin{itemize}
\item $Q_z=\frac12 Q(2s_z, y_z), \qquad Q_z^{*}=Q(2s_z,y_z),\qquad Q_z^{**}=Q(4 c_1 s_z, y_z)$\\
\end{itemize}
\end{easylist}
%%%%%%%%%%%%
Since $2c_1\tilde{3}\,\sigma\leq c_2\,s_z$, Lemma~\ref{lem:bigcubes} implies that $Q^{**}_z\subset Q_{b,b}$ in all the above cases, as otherwise there is a contradiction to $\lambda>\mu_{a,b}$ by \eqref{eq:lambda}. 

On the one hand, \eqref{scal:7} of Lemma~\ref{lem:scal} implies that $Q_z^{**}\supset Q_z
\supset Q(s_z,y_z)$, and therefore, we conclude that
\[
\dashiint_{Q_z^{**}} F\,dxdt\leq 2 \lambda.
\]
On the other hand, since \eqref{scal:3} of Lemma~\ref{lem:scal} implies that
$\abs{Q_z}\approx \abs{Q_z^{**}}\approx \abs{Q(s_z,y_z)}$, by \eqref{eq:lambda} we find that
\begin{align}
\label{lambda:low}
\lambda \leq c \dashiint_{Q_z} F\,dxdt.
\end{align}
The proof of the reverse H\"older inequality has to be split in several sub-cases.
\vskip.2truecm
\noindent{\bf Case 1.}\\
In this case $Q_z$ is intrinsic and the degenerate alternative holds.
% In order to apply Proposition~\ref{deg}, we have to check that $2Q_z$ is sub-intrinsic. To prove this, firstly observe  $2Q_z\subset Q(\tilde{2}\rho,y_z)$, which is sub-intrinsic by construction. Secondly, since $2Q$ and $Q(\tilde{2}\rho,y_z)$ have comparable measure by \eqref{scal:4} of Lemma~\ref{lem:scal1}, we got that $2Q_z$ is sub-intrinsic. 
Consequently, \eqref{lambda:low} and Proposition~\ref{Prop:rev-hold-deg} imply
\[
\lambda \leq c\dashiint_{Q_z}  F\,dxdt\leq  C\bigg(\dashiint_{Q_z^*}F^q\,dxdt\bigg)^\frac1q+ C M^*(\tilde f^2)(y),
\]
 for any $y\in Q_z^*$.
\vskip.2truecm
\noindent{\bf Case 2.}\\ 
In this case $Q_z^*$ is intrinsic and the non-degenerate alternative holds. Proposition~\ref{Prop:rev-hol-non-deg} and \eqref{lambda:low} directly imply
\[
\lambda \leq c\dashiint_{Q_z}  F\,dxdt\leq  C\bigg(\dashiint_{Q_z^*}F^q\,dxdt\bigg)^\frac1q+ C M^*(\tilde f^2)(y),
\]
 for any $y\in Q_z^*$.
\vskip.2truecm
\noindent{\bf Case 3.}\\
This is the most delicate part.
Applying Corollary~\ref{cor:deg-subint} on the cylinder $Q_z=\frac12Q(2s_z,y_z)$ yields
\[
\lambda \leq c\, \dashiint_{Q_z}  F\,dxdt\leq  \gamma_1\frac{\theta_{2 s_z,y_z}^{\frac{m+1}{1-m}}}{2s_z}+ M^*(\tilde f^2)(y).
\]
We now show that
\begin{align}
\label{subintr:3}
\gamma_1\frac{\theta_{2 s_z,y_z}^{\frac{m+1}{1-m}}}{2s_z}\leq \frac12\lambda+ M^*(\tilde f^2)(y).
\end{align}
Once this is proven, then by absorption, estimate \eqref{eq:ii} follows.

We begin by defining $\sigma_z:=\inf\set{s\in [c_o s_z,1] : Q(s,y_z)\text{ is intrinsic}}$. Since we are in Case 3, we find that $\sigma_z\in (c_o s_z,1]$; \eqref{scal:6} of Lemma~\ref{lem:scal} and the choice of $c_o$ in \eqref{eq:c1} imply that
\begin{align*}
\theta_{2 s_z,y_z}&\le\left(\frac{2 s_z}{\sigma_z}\right)^\beta\theta_{\sigma_z,y}\\
\Rightarrow\,\,\gamma_1\frac{\theta_{2 s_z,y_z}^{\frac{m+1}{1-m}}}{2 s_z}&\le\gamma_1\left(\frac{2 s_z}{\sigma_z}\right)^{\frac{m+1}{1-m}\beta}\frac{\theta_{\sigma_z,y_z}^{\frac{m+1}{1-m}}}{2 s_z}\\
\Rightarrow\,\,\gamma_1\frac{\theta_{2 s_z,y_z}^{\frac{m+1}{1-m}}}{2 s_z}&\le\gamma_1\left(\frac{2 s_z}{\sigma_z}\right)^{\frac{m+1}{1-m}\beta-1}\frac{\theta_{\sigma_z,y_z}^{\frac{m+1}{1-m}}}{\sigma_z}\\
\Rightarrow\,\,\gamma_1\frac{\theta_{2 s_z,y_z}^{\frac{m+1}{1-m}}}{2 s_z}&\le\gamma_1\left(\frac{2}{c_o}\right)^{\frac{m+1}{1-m}\beta-1}\frac{\theta_{\sigma_z,y_z}^{\frac{m+1}{1-m}}}{\sigma_z}.
\end{align*}
Notice that in \eqref{b} we fixed $\beta$ large enough (that is, $\hat{b}$ small enough) such that
\[
1<\frac{m+1}{1-m}\beta=\frac{m+1}{1-m}(1-\hat{b})
\]
and in \eqref{eq:c1} $c_o$ is such that 
\[
\gamma_1\left(\frac{2}{c_o}\right)^{\frac{m+1}{1-m}\beta-1}=\tilde{\delta}.
\]
We conclude that
\begin{align}
\label{subintr:2}
\gamma_1\frac{\theta_{2 s_z,y_z}}{2s_z}\leq \tilde{\delta} \frac{\theta_{\sigma_z,y_z}}{\sigma_z}.
\end{align}
In the simple case $\sigma_z\in ({1}/{\tilde 3},1]$, we further estimate by \eqref{scal?:1} of 
Lemma~\ref{lem:scal?}. Indeed, since $\lambda>1$ and $\theta_o\sim 1$, by \eqref{startcylinder} we find  that there is a $c_3\geq 1$ independent of $\tilde{\delta}$, such that
\[
\frac{\theta_{\sigma_z,y_z}}{\sigma_z}\leq c\theta_{1,y_z}\leq c_3\lambda;
\]
this implies \eqref{subintr:3} once we choose $\tilde{\delta}\leq\frac{1}{2c_3}$.

\noindent In the difficult case $\sigma_z\in (c_o s_z,{1}/{\tilde 3}]$, we use the fact that $Q(\sigma_z,y_z)$ is intrinsic. Hence by Lemma~\ref{lem:scal}, \eqref{scal:2} and \eqref{scal:6} we find that
\[
\dashiint_{Q(\sigma_z/2,y_z)}\tilde{u}^{m+1}\,dxdt\leq \theta_{\sigma_z/2,y_z}^\frac{m+1}{1-m}\leq \Big(\frac{\theta_{\sigma_z,y_z}}{2}\Big)^\frac{m+1}{1-m}\leq \frac{1}{2^\frac{m+1}{1-m}}\dashiint_{Q(\sigma_z,y_z)}\tilde{u}^{m+1}\,dxdt.
\]
By Lemma~\ref{seb-app}, this implies that $Q(\sigma_z,y_z)$ is degenerate and intrinsic. Now Proposition~\ref{Prop:rev-hold-deg}, \eqref{scal:7} of Lemma~\ref{lem:scal}, and Jensen's inequality imply
\[
\tilde{\delta}\, \frac{\theta_{\sigma_z,y_z}}{\sigma_z}\leq c\tilde{\delta}\bigg(\dashiint_{3Q(\sigma_z,y_z)} F^q\,dxdt\bigg)^\frac{1}{q}+M^*(\tilde f^2)(y)
\leq C\,\tilde{\delta}\dashiint_{Q(\tilde{3}\sigma_z,y_z)} F\,dxdt+C M^*(\tilde f^2)(y).
\]
Therefore, by the construction of $Q(s_z,y_z)$, we find that there is a $c_4$ independent of $\tilde{\delta}$, such that
\[
\tilde{\delta} \frac{\theta_{\sigma_z,y_z}}{\sigma_z} \leq c_4\tilde{\delta}\lambda +M^*(\tilde f^2)(y).
\]
Now \eqref{subintr:3} follows by choosing
\begin{align}
\label{eq:delta}
\tilde{\delta}=\min\set{\frac{1}{2c_3},\frac{1}{2c_4}}.
\end{align}
This finishes the construction of $Q_z,Q_z^*$ and $Q_z^{**}$. 
The covering now follows by Lemma~\ref{lem:scal?} and Lemma~\ref{Lm:Vitali:1}, for $\Omega=O_\lambda$ and $U(z,s_z):=Q_z^*$. 
\end{proof}
We can finally conclude and prove the higher integrability result.
\begin{thm}[Intrinsic]\label{thm-intr}
Let $u\ge0$ be a local, weak solution to {\eqref{PME}-\eqref{PMS-eq:structure}} in the space-time cylinder $E_T$ for $m \in \Big(\frac{(n-2)_+}{n+2},1\Big)$. There exist an exponent $p>1$, and a constant $c$ that depends only on the data, such that for any sub-intrinsic parabolic cylinder 
\[
\bigg(\dashiint_{Q_{S,\sqrt{S/\theta_o}}} u^{m+1}\,dxdt\bigg)^\frac{1-m}{m+1}\leq C\theta_o,
\]
we have
\begin{align*}
\bigg(\dashiint_{Q_{\frac12S,\frac12\sqrt{S/\theta_o}}}\abs{D u^m}^{2p}\,dxdt\bigg)^\frac{1}{p}\leq c\bigg(\dashiint_{Q_{2S,2\sqrt{S/\theta_o}}}f^{2p}\,dxdt\bigg)^\frac{1}{p}+c'\,\frac{\theta_o^\frac{m+1}{1-m}}{S},
\end{align*}
where $c'$ additionally depends on $C$.
\end{thm}
%%%%%%%%%%%%%%%%%%%%
\begin{proof}
We define the so-called \emph{bad set}
\begin{align*}
U_\lambda :=O_\lambda\cap \set{M^*(\tilde f^2\chi_{Q_{2,2}})\leq \tilde\epsilon\lambda},
\end{align*}
for some $\tilde\epsilon$, which will be chosen later.
We proceed by providing a re-distributional estimate. We take $\frac12\leq a<b\leq 1$ and the corresponding covering constructed in Lemma~\ref{lem:czcover}.
We start by
  \begin{align*}
  \abs{Q_i}=\abs{Q_i\cap U_\lambda}+\abs{Q_i\cap(U_\lambda)^c},
  \end{align*}
where $Q_i$ is one of the cylinders $Q_{z_i}$ built in Proposition~\ref{lem:czcover}.
Let us first consider the case 
\[
\frac{\abs{Q_i\cap U_\lambda}}{\abs{Q_i}}\geq \frac12.
\] 
This implies that there exists $y\in  Q_i$, such that $M^*(f)(y)\leq \tilde\epsilon \lambda$.
	We can apply the reverse 
  H\"older estimate \eqref{eq:ii} of Proposition~\ref{lem:czcover}, and obtain for some $\gamma\in(0,1)$ that
  \begin{align*}
  \lambda^q\leq c\left(\dashiint_{Q_i} F\,dxdt\right)^q
  \leq \frac{C}{\abs{Q_i}}\iint_{Q_i^{*}}F^q\chi_{\set{F>\gamma\lambda}}\,dxdt+C(\gamma\lambda)^q+C\tilde\epsilon\lambda^q.
  \end{align*}
We now choose $\gamma$ and $\tilde\epsilon$ conveniently small, such that $C(\gamma\lambda)^q+C\tilde\epsilon\lambda^q=\frac12\lambda^q$ and find
 \begin{align*}
\lambda\abs{Q_i}\leq c\lambda^{1-q}\iint_{Q_i^{*}}F^q\chi_{\set{F>\gamma\lambda}}\,dxdt.
  \end{align*}
On the other hand,
\[
\abs{Q_i}\leq 2\abs{Q_i\cap (U_\lambda)^c}\quad\Rightarrow\quad\lambda\abs{Q_i}\leq 2\lambda\abs{Q_i\cap (U_\lambda)^c}.
\] 
Therefore, in any case,
\[
\lambda\abs{Q_i}\leq C \lambda^{1-q}\iint_{Q_i^{*}}F^q\chi_{\set{F>\gamma\lambda}}\,dxdt+2\lambda\abs{Q_i\cap (U_\lambda)^c}.
\]
We proceed by using the last estimates as well as the fact that $(Q_i^{**})_i$ covers the set $O_\lambda\cap Q_{a,a}$.
  \begin{align*}
  \iint_{Q_{a,a}\cap O_\lambda} F\,dxdt&\leq \sum_i\iint_{Q_i^{**}}F\,dxdt\leq 2\lambda\sum_i\abs{Q_i^{**}}\leq c\lambda\sum_i\abs{Q_i}\\
  &\leq C \lambda^{1-q}\sum_i\iint_{Q_i^*}F^q\chi_{\set{F>\gamma\lambda}}\,dxdt+2 c \lambda\abs{Q_i\cap(U_\lambda)^c}\\
  &\leq C \lambda^{1-q}\iint_{Q_{b,b}}F^q\chi_{\set{F>\gamma\lambda}}\,dxdt+2 c \lambda\abs{Q_{b,b}\cap \set{M^*(\tilde f^2\chi_{Q_{2,2}})>\tilde\epsilon\lambda}}.
  \end{align*}
We pick $\alpha\in (0,1)$ to be chosen later, $k\in\setN$, multiply the above estimate by $\lambda^{-\alpha}$, and integrate from $\mu_{a,b}$ to $k$ with respect to $\lambda$. This implies
%%%%%%%%%%%%%%%%%%%%%  
\begin{align*}
 (I)&:= \int_{\mu_{a,b}}^k\lambda^{-\alpha}\iint_{Q_{a,a}\cap O_\lambda} F\,dxdt\,d\lambda
  \leq {C}\int_{{\mu_{a,b}}}^k\lambda^{1-q-\alpha}\iint_{Q_{b,b}}F^q\chi_{\set{F>\gamma\lambda}}\,dxdt\,d\lambda\\
  &\quad+2{c}\int_{{\mu_{a,b}}}^k\lambda^{1-\alpha}\abs{Q_{b,b}\cap\set{M^*({\tilde f^2})>\tilde\epsilon\lambda}}d\lambda
 =:(II)+(III).
  \end{align*}
%%%%%%%%%%%%%%%%%%%%%  
We estimate from below
  \begin{align*}
(I)&\geq \int_0^k\lambda^{-\alpha}\iint_{Q_{a,a}} F\chi_{\set{F>\lambda}}\,dxdt\,d\lambda-\frac{{\mu_{a,b}}^{2-\alpha}}{1-\alpha}\\
&=\iint_{Q_{a,a}}F\int_0^{\min\set{F(x),k}}\lambda^{-\alpha}d\lambda\,dxdt-\frac{{\mu_{a,b}}^{2-\alpha}}{1-\alpha}\\
&=\frac{1}{1-\alpha}\iint_{Q_{a,a}} F\min\set{F,k}^{1-\alpha}\,dxdt -\frac{{\mu_{a,b}}^{2-\alpha}}{1-\alpha}.
  \end{align*}
The bound from above is analogous
  \begin{align*}
  (II)&\leq {C}\int_{\gamma{{\mu_{a,b}}}}^{\gamma k} \lambda^{1-q-\alpha}\iint_{Q_{b,b}}F^q\chi_{\set{F>\lambda}}\,
  dxdt\,d\lambda\\
  &\leq \frac{{C}}{2-q-\alpha}\iint_{Q_{b,b}} F^q\min\set{F,k}^{2-q-\alpha}\,dxdt.
  \end{align*}
Finally, $(III)$ is estimated by the continuity of the maximal function and the classical integral representation via level sets. We calculate and estimate
  \begin{align*}
  (III)&= \int_{{\mu_{a,b}}}^k\lambda^{1-\alpha}\abs{Q_{b,b}\cap\set{M^*({\tilde f^{2}})>\tilde\epsilon\lambda}}d\lambda
	\leq c\iint_{Q_{b,b}} \abs{M^*({\tilde f^{2}}Q_{2,2})})^{2-\alpha}\,dxdt\\
	&\leq c\iint_{Q_{2,2}} \tilde{f}^{{2(2-\alpha)}}\,dxdt.
  \end{align*}
  All together, using the definition of ${\mu_{a,b}}$ from \eqref{eq:subintr}, we find that
  \begin{align*}
  \frac{1}{1-\alpha}\iint_{Q_{a,a}} F\min\set{F,k}^{1-\alpha}\,dxdt& \leq \frac{c}{2-q-\alpha}\iint_{Q_{b,b}} F^q\min\set{F,k}^{2-q-\alpha}\,dxdt \\
  &\quad+ \frac{c\abs{b-a}^{-\tau(2-\alpha)}}{1-\alpha}
   + c\iint_{Q_{2,2}} \tilde{f}^{{2(2-\alpha)}}\,dxdt.
  \end{align*}
  Now, we fix $\alpha\in (0,1)$ in such a way, that
  \[
  \frac{c(1-\alpha)}{2-q-\alpha}\leq \frac12.
  \]
  This implies
    \begin{align*}
 \iint_{Q_{a,a}} F\min\set{F,k}^{1-\alpha}\,dxdt& \leq \frac{1}{2}\iint_{Q_{b,b}} F\min\set{F,k}^{1-\alpha}\,dxdt \\
  &\quad+ c\abs{b-a}^{-\tau(2-\alpha)}
   + c\iint_{Q_{2,2}} \tilde{f}^{{2(2-\alpha)}}\,dxdt.
  \end{align*}
  Finally, the interpolation Lemma~6.1 of \cite{Giu03} implies that for every $k\in \setN$
  \begin{align*}
 \iint_{Q_{\frac12,\frac12}} F\min\set{F,k}^{1-\alpha}\,dxdt& \leq c
   + c\iint_{Q_{2,2}} \tilde{f}^{{2(2-\alpha)}}\,dxdt.
  \end{align*}
Letting $k\to\infty$ for $p=2-\alpha>1$ yields that
	\begin{align*}
\dashiint_{Q_{\frac12,\frac12}}\abs{D \tilde{u}^m}^{2p}dxdt\leq  c
   + c\,\dashiint_{Q_{2,2}} \tilde{f}^{{2p}}\,dxdt.
\end{align*}
This implies the desired result by scaling back to $u$.
%\begin{align}
%\dashint_{Q_{\gamma^{m-1}\rho^2,\rho}}\abs{D u^\frac{m+1}{2}}^{2p}\leq c\bigg(\dashint_{Q_{\gamma^{m-1}4\rho^2,2\rho}}\abs{D u^\frac{m+1}{2}}^{2}\bigg)^p+\frac{c}{(\gamma^{m+1}\rho^2)^{p-1}}\dashint_{Q_{\gamma^{m-1}4\rho^2,2\rho}} \frac{u^{m+1}}{\rho^2},
%\end{align}
%now Young's inequality implies the result.
 \end{proof}
	\begin{thm}[parabolic]
\label{thm-para}
Let $u\ge0$ be a local, weak solution to {\eqref{PME}-\eqref{PMS-eq:structure}} in the space-time cylinder $E_T$ for $m \in \Big(\frac{(n-2)_+}{n+2},1\Big)$. There exist an exponent $p>1$ and a constant $c$, that depends only on the data, such that for any parabolic cylinder $Q_{R^2,R}\subset E_T$ with
\[
{\bigg(\dashiint_{Q_{R^2,R}} u^{m+1}\,dxdt\bigg)^{1-m}=K,}
\]
we have
\begin{align}\label{eq:par-final}
\bigg(\dashiint_{\frac12Q_{R^2,R}}\abs{D u^m}^{2p}\,dxdt\bigg)^\frac{1-m}{2mp}\leq c\sqrt{K}\bigg(\dashiint_{Q_{R^2,R}}{R^{2p}}{f^{2p}}\,dxdt\bigg)^\frac{1-m}{2mp}+cK^\frac{3}{2}+c.
\end{align}
\end{thm}
\begin{proof}
Estimate \eqref{eq:par-final} is proved by covering $Q_{R^2,R}$ with proper sub-intrinsic cylinders. As {\eqref{PME}-\eqref{PMS-eq:structure}} is essentially invariant under the classical parabolic scaling, without loss of generality, we may assume $R=1$ .

We define $K$ as the number for which
\[
\bigg(\iint_{Q_{1,1}} u^{m+1}\, dxdt\bigg)^{1-m}=K.
\]
Now let $r\in (0,1]$ and $s\in (0,1]$;  any $Q_{s,r}\subset Q_{1,1}$ satisfies
\[
\bigg(\iint_{Q_{s,r}} u^{m+1} \,dxdt\bigg)^{m-1}\leq K.
\]  
If $\frac{r^{2(m+1)}\abs{B_{r}}^{m-1}}{s^2}\geq K$, then the cylinder is sub-intrinsic, since
\[
\bigg(\iint_{Q_{s,r}}u^{m+1} \,dxdt\bigg)^{m-1}\leq \frac{r^{2(m+1)}\abs{B_{r}}^{m-1}}{s^2}.
\]  
If $K\leq 1$, we can pick $s,r=1$ and the result follows by Theorem~\ref{thm-intr}. If $K>1$, we choose $r=1$ and $s=\frac{1}{\sqrt{K}}$. We can then cover $Q_{1,1}$ by  $N$ sub-cylinders of the above type, where 
$$\left\lfloor\sqrt{K}\right\rfloor\leq  N \leq \left\lceil\sqrt{K}\right\rceil,$$
and $\lfloor\cdot\rfloor$, $\lceil\cdot\rceil$ are the floor and ceiling functions, respectively.
This concludes the proof by Theorem~\ref{thm-intr}.
\end{proof}

\providecommand{\bysame}{\leavevmode\hbox to3em{\hrulefill}\thinspace}
\providecommand{\MR}{\relax\ifhmode\unskip\space\fi MR }
\providecommand{\MRhref}[2]{%
  \href{http://www.ams.org/mathscinet-getitem?mr=#1}{#2}
}
\providecommand{\href}[2]{#2}

\end{document}